\documentclass[11pt,fleqn,a4paper ]{article} 

\usepackage{amsmath,amssymb,amsthm,amsfonts,tikz-cd,enumitem} 
\usepackage[mathscr]{eucal}
\usepackage[all]{xy}

\usepackage{hyperref}
\hypersetup{colorlinks, linkcolor=blue, citecolor=blue, urlcolor=blue}

\allowdisplaybreaks

\makeatletter
\long\def\@makecaption#1#2{%
	\vskip\abovecaptionskip\footnotesize
	\sbox\@tempboxa{#1. #2}%
	\ifdim \wd\@tempboxa >\hsize
	#1. #2\par
	\else
	\global \@minipagefalse
	\hb@xt@\hsize{\hfil\box\@tempboxa\hfil}%
	\fi
	\vskip\belowcaptionskip}
\makeatother

\newcommand{\p}{\partial}

\newcommand{\sgn}{\mathop{\rm sgn}\nolimits}

\newtheorem{theorem}{Theorem}
\newtheorem{lemma}[theorem]{Lemma}
\newtheorem{corollary}[theorem]{Corollary}
\newtheorem{conjecture}[theorem]{Conjecture}
\newtheorem{question}{Question}
\newtheorem{proposition}[theorem]{Proposition}
{\theoremstyle{definition}
	\newtheorem{definition}[theorem]{Definition}
	\newtheorem{remark}[theorem]{Remark}
	\newtheorem{example}[theorem]{Example}
}

\newcommand{\todo}[1][\null]{\ensuremath{\clubsuit}}

\newcommand{\noprint}[1]{}

\begin{document}
\setlist[description]{font=\normalfont}

\par\noindent {\LARGE\bf
Generic graded contractions of Lie algebras
\par}

\vspace{5.5mm}\par\noindent{\large
Mikhail V. Kochetov$^\dag$ and Serhii D. Koval$^{\dag\ddag}$
}

\vspace{5.5mm}\par\noindent{\it\small
$^\dag$Department of Mathematics and Statistics, Memorial University of Newfoundland,\\
$\phantom{^\dag}$\,St.\ John's (NL) A1C 5S7, Canada
\par}

\vspace{2mm}\par\noindent{\it\small
$^\ddag$Institute of Mathematics of NAS of Ukraine, 3 Tereshchenkivska Str., 01024 Kyiv, Ukraine
\par}

\vspace{5mm}\par\noindent
E-mails:
mikhail@mun.ca,
skoval@mun.ca


\vspace{6mm}\par\noindent\hspace*{10mm}\parbox{150mm}{\small
We study generic graded contractions of Lie algebras
from the perspectives of group cohomology, affine algebraic geometry and monoidal categories.
We show that generic graded contractions with a fixed support
are classified by a certain abelian group, which we explicitly describe.
Analyzing the variety of generic graded contractions as an affine algebraic variety
allows us to describe which generic graded contractions 
define graded degenerations of a given graded Lie algebra.
Using the interpretation of generic $G$-graded contractions
as lax monoidal structures on the identity endofunctor
of the monoidal category of $G$-graded vector spaces,
we establish a functorial version of the Weimar--Woods conjecture on equivalence of generic graded contractions.
}\par\vspace{2mm}

\noprint{
Keywords:
Graded Lie algebras;
Graded contractions;
Group cohomology;
Affine varieties;
Binomial ideals;
Lax monoidal functors.

MSC:

17Bxx  Lie algebras and Lie superalgebras 
     17B70  Graded Lie algebras
     17B05  Structure theory for Lie algebras and superalgebras
     17B81  Applications of Lie (super)algebras to physics, etc.
     
18Mxx  Monoidal categories and operads
     18M05  Monoidal categories, symmetric monoidal categories
     
20Jxx  Connections of group theory with homological algebra and category theory
     20J06  Cohomology of groups
}

\section{Introduction}~\label{sec:Introduction}

Deformation theory of different algebraic and geometric objects
has been a subject of significant interest in mathematics
ever since the study of Riemann surfaces and their moduli spaces by Riemann in the XIX century,
see a historical review of deformation theory in~\cite{fial2008a}.
The inverse process to deformation is that of degeneration~\cite{levy1967a} and, in particular, contraction.
A well-known example from physics is the contraction
of Poincar\'e group of special relativity to Galilean group of classical mechanics
when the speed of light is assumed to approach infinity.
This example motivated Segal~\cite{sega1951a} to initiate the study 
of contractions of Lie groups and Lie algebras.
Later, In\"on\"u and Wigner~\cite{inon1953a,inon1954a} considered the simplest types of one-parameter continuous contractions,
where some of the basis elements of the Lie algebra are scaled by a parameter that tends to zero.
This type of contractions is now called the {\it In\"on\"u--Wigner contractions}.
Saletan~\cite{sale1961a} and then Doebner and Melsheimer~\cite{doeb1967a}
generalized this concept in several directions, leading to what are now known as
{\it generalized In\"on\"u--Wigner contractions}.

Contractions have found applications in the study of representations of Lie groups~\cite{dool1985a} and Lie algebras~\cite{catt1979a,weim1991a},
invariants~\cite{weim2008a} and special functions~\cite{kaln1999a}.
For a more comprehensive literature review on this topic,
as well as the state of the art on contractions between low-dimensional real and complex Lie algebras,
see, e.g.,~\cite{nest2006a}.
Contractions of some infinite-dimensional Lie algebras (Witt, Virasoro and affine Kac--Moody algebras)
were studied in~\cite{fial2005a,fial2006a}.

Contractions of Lie algebras are closely related to degenerations.
In fact, over the field of complex numbers these two concepts coincide due to Borel's closed orbit lemma.
Milnor, in his seminal work on invariant metrics on Lie groups~\cite{miln1976a},
used degenerations of Lie algebras to demonstrate the existence of invariant Riemannian metrics on Lie groups.
Essentially, degenerations of an $n$-dimensional Lie algebra $\mathfrak L$
are limit points of its ${\rm GL}_n(\mathbb F)$-orbit $\mathcal O(\mathfrak L)$ in the variety of $n$-dimensional Lie algebras
endowed with the Zariski topology.
When the orbit $\mathcal O(\mathfrak L)$ is a Zariski open set,
the Lie algebra $\mathfrak L$ is called {\it rigid}.
Intuitively, a Lie algebra is rigid if all Lie algebras close to it
(in the sense of Zariski topology) are isomorphic.
Rigidity of Lie algebras has been extensively studied
and characterized in different terms,
see, e.g.,~\cite{carl1984a,goze2006a} and references therein. 
The classification of rigid Lie algebras remains an important open problem.

While the general linear group ${\rm GL}_n(\mathbb F)$ is the standard choice to study degenerations, 
in the context of Lie algebras graded by an abelian group~$G$, it is natural to replace it with an algebraic torus,
see Subsection~\ref{sec:GradingsAndContractions}.
Montigny and Patera in~\cite{mont1991a} defined a new type of Lie algebra contractions,
which are in general multi-parametric and not necessarily continuous.
According to the authors, this provides ``a much larger variety of contraction `limits' to be studied''.
This type of contractions is known as {\it graded contractions}, and its defining feature is the preservation
of a chosen $G$-grading during the contraction process.
In Section~\ref{sec:VarietyOfContractions},
we investigate which among these contractions are graded degenerations (Corollary~\ref{cor:GradedDegenerationsAndContractions}) and,
in particular, related to In\"on\"u--Wigner contractions.

Over the past three decades, graded contractions have been extensively studied from various perspectives,
including computations for specific graded Lie algebras \cite{beru1993a,cout1991a,drap2024b, drap2024a,hriv2006a,hriv2013a,kost2001a,mont1996a} 
and theoretical aspects \cite{mood1991a,weim2006a,weim2006b}.
The concept of graded contraction is flexible with respect to changes in the underlying category of algebras
and thus can be easily transferred to other structures,
such as superalgebras~\cite{mont1991a} and Jordan algebras~\cite{kash2003a}.
For a review of the state of the art in graded contractions, we refer to the recent paper by Draper, S\'anchez-Ortega and Meyer~\cite{drap2024a}.

For generic graded contractions,
which are studied in \cite{mood1991a,weim1995a,weim2006a} and are also the subject of the present paper,
the defining conditions are chosen in such a way
that the contraction can be applied uniformly to all $G$-graded Lie algebras.
This allows us to study such contractions from various perspectives:
group cohomology, affine algebraic geometry and monoidal categories.
As a result, we not only generalize and reinterpret existing results on classification of such contractions from~\cite{mood1991a,weim2006a},
but also place them into a wider context.

The structure of the paper is as follows.
In Section~\ref{sec:Preliminaries}, we briefly review the necessary background on cohomology of abelian groups as well as  gradings,
contractions, degenerations, and graded contractions.

In Section~\ref{sec:GroupCohomology}, we discuss the classification of generic graded contractions.
To any subset $S$ of the set of unordered pairs in $G$,
we associate the free abelian groups of
``surviving defining relations'' and ``surviving higher-order identities''
and use them to show that contractions with support~$S$ 
are classified by a certain abelian group ${\rm H}^2_S(\mathbb F^\times)$,
which can be computed using the exact sequences presented in Theorem~\ref{thm:ClassificationOfContractions}.
In particular, for an algebraically closed or a real closed field $\mathbb F$,
we reinterpret the results of Weimar--Woods~\cite{weim2006a}.

Section~\ref{sec:VarietyOfContractions} is devoted to the affine variety of generic graded contractions.
We study the action of an algebraic torus on this variety
and the orbit closures of contractions.
In Theorem~\ref{thm:VanishingIdeal}, we describe the vanishing ideal of such an orbit closure.
We show that the contractions in the closure of the unique open orbit
are precisely the ``continuous contractions'' in the sense of~\cite{mood1991a}
and relate them to In\"on\"u--Wigner contractions
using Theorem~\ref{thm:ClassificationOfContractions} and~\cite[Theorem~VI.2]{weim1995a}.

In Section~\ref{sec:EquivOfContractions}, we discuss the Weimar--Woods conjecture that {\it two graded contractions are isomorphic
if and only if they are equivalent via normalization.}
We show that a generic $G$-graded contraction endows the identity endofunctor
of the monoidal category of $G$-graded vector spaces with a lax monoidal structure.
This allows us to prove a functorial version of the Weimar--Woods conjecture:
{\it two graded contractions define naturally isomorphic lax monoidal endofunctors
if and only if they are equivalent via normalization} (Proposition~\ref{prop:LaxWeimWoodsConj}).

Finally, in Section~\ref{sec:Conclusion},
we summarize the results of the paper and discuss a few possible directions for future work.

\section{Preliminaries}\label{sec:Preliminaries}

In this section, we review some basic facts about cohomology and extensions of abelian groups,
see, e.g., \cite[Chapter~I]{macl1963A} and \cite[Chapter~6]{weib1994A},
contractions~\cite{nest2006a}, gradings~\cite{eldu2013A}
and graded contractions of Lie algebras~\cite{drap2024a,mont1991a}.\looseness=-1

\subsection{Cohomology and extensions of abelian groups}\label{sec:PrelimGroupCohomology}

Let $R$ be an associative unital ring and let $A$ and $B$ be $R$-modules.
An {\it extension} of $A$ by $B$ is an $R$-module $E$
containing a submodule isomorphic to $B$ such that $E/B\simeq A$.
In other words, the modules $A$, $B$ and $E$ fit into the short exact sequence
\[
\begin{tikzcd}
B
\arrow[r,hook] 
&E\arrow[r,twoheadrightarrow]
&A,
\end{tikzcd}
\]
where the arrows $\hookrightarrow$ and $\twoheadrightarrow$ denote injective and surjective $R$-module homomorphisms, respectively.
Two extensions $E$ and $E'$ are called equivalent if there exists an $R$-module homomorphism $\varphi\colon E\to E'$
such that the diagram
\[
\begin{tikzcd}
B
\arrow[r,hook] 
\arrow[d,equal] 
&E\arrow[d,"\varphi"]\arrow[r,twoheadrightarrow]
&A\arrow[d,equal] 
\\
B
\arrow[r,hook] 
&E'\arrow[r,twoheadrightarrow]
&A
\end{tikzcd}
\]
is commutative.
It follows that $\varphi$ is an isomorphism of $R$-modules,
and this gives rise to an equivalence relation on the set of extensions.
The set of equivalence classes of extensions of $A$ by $B$ is denoted by ${\rm Ext}_R(A,B)$
or, simply, ${\rm Ext}(A,B)$ when $R$ is clear from the context.

For two extensions $E$ and $E'$ of $A$ by $B$ one can form the so-called Baer's sum $E+E'$ of extensions~\cite[Chapter III.2]{macl1963A},
which defines a binary operation on the set ${\rm Ext}(A,B)$
that makes it an abelian group~\cite[Chapter III.2, Theorem~2.1]{macl1963A}.
The neutral element of this group is the equivalence class of the direct sum $A\oplus B$;
the elements of this class are called {\it split extensions}.
In particular, all extensions are split when $A$ is projective or $B$ is injective as $R$-modules,
\cite[Sections~2.2 and~2.3]{weib1994A}.
In the case of abelian groups, i.e., $R=\mathbb Z$, this means that $A$ is free or $B$ is divisible.

Let $A$ and $B$ be abelian groups.
The group ${\rm Ext}(A,B)$ is isomorphic to the symmetric part of the second cohomology group ${\rm H}^2(A,B)$
of $A$ with coefficients in $B$, considered as an $A$-module with the trivial $A$-action, see~\cite[Theorem~6.6.3]{weib1994A}.
More specifically, the abelian group ${\rm H}^2_{\rm sym}(A,B)$ is the quotient of the group of symmetric two-cocycles ${\rm Z}^2_{\rm sym}(A,B)$
of $A$ with values in $B$ by the group of coboundaries ${\rm B}^2(A,B)$,
\[
{\rm H}^2_{\rm sym}(A,B):={\rm Z}^2_{\rm sym}(A,B)/{\rm B}^2(A,B).
\]
A {\it symmetric two-cocycle} $\varepsilon$ is a map from $A\times A$ to $B$ satisfying $\varepsilon(a_1,a_2)=\varepsilon(a_2,a_1)$
and
\[
\varepsilon(a_1,a_2)\varepsilon(a_1a_2,a_3)=\varepsilon(a_1,a_2a_3)\varepsilon(a_2,a_3)
\quad\text{for all}\quad
a_1,a_2,a_3\in A,
\]
where the group operations in $A$ and $B$ are written multiplicatively.
For an arbitrary map $\alpha\colon A\to B$, its {\it coboundary} $d\alpha$ is the symmetric two-cocycle given by
\[
d\alpha(a_1,a_2):=\alpha(a_1)\alpha(a_2)\alpha(a_1a_2)^{-1}
\quad\text{for all}\quad
a_1,a_2\in A.
\]
The correspondence between the groups ${\rm Ext}(A,B)$ and ${\rm H}^2_{\rm sym}(A,B)$ is constructed as follows.
For a group $E$ in a short exact sequence \smash{$B\stackrel{\raisebox{-1.2ex}[0ex][0ex]{\scriptsize$\iota$}}{\hookrightarrow}E\stackrel{\raisebox{-1.2ex}[0ex][0ex]{\scriptsize$\pi$}}{\twoheadrightarrow}A$}, consider a set-theoretic section $\sigma\colon A\to E$ of $\pi$.
Then the element $\sigma(a_1)\sigma(a_2)\sigma(a_1a_2)^{-1}$ belongs to $\ker\pi$,
thus by exactness of the sequence there exists a unique $\varepsilon(a_1,a_2)\in B$ such that $\iota(\varepsilon(a_1,a_2))=\sigma(a_1)\sigma(a_2)\sigma(a_1a_2)^{-1}$
for all $a_1,a_2\in A$.
It is easy to verify that the map $\varepsilon\colon A\times A\to B$ is a two-cocycle,
which is symmetric if and only if $E$ is abelian.
The equivalence class $[\varepsilon]\in{\rm H}^2(A,B)$
does not depend on the choice of the section~$\sigma$.

A common approach to computing ${\rm Ext}(A,B)$ relies on finding a projective resolution of $A$
and then computing ${\rm Ext}(A,B)$ as the first (left) derived functor of the functor ${\rm Hom}(\cdot,B)$,
see, e.g., \cite[Chapter~III]{macl1963A}.
In the case $R=\mathbb Z$, all subsequent derived functors are trivial.
The choice of the projective resolution does not affect the group ${\rm Ext}(A,B)$
and can be adapted to the group~$A$, which will be useful in Section~\ref{sec:GroupCohomology}.

\subsection{Contractions and degenerations of Lie algebras}\label{sec:ContAndDegenOfLieAlgs}

\par\noindent
{\it Contractions}.
Let $\mathfrak L:=(V,[\cdot,\cdot])$ be a real or complex Lie algebra
with an underlying vector space~$V$ and a Lie bracket $[\cdot,\cdot]$.
Consider a continuous function $U\colon(0,1]\to{\rm GL}(V)$.
A parameterized family of new Lie brackets on $V$ is defined in terms of the Lie bracket $[\cdot,\cdot]$ in the following way:
\begin{gather}\label{eq:ParameterizedLieBracket}
\forall t\in(0,1],\ \
\forall x,y\in V\colon\quad
[x,y]_t:=U_t^{-1}[U_tx,U_t y].
\end{gather}
It is clear that for any $t\in(0,1] $ the Lie algebra
$\mathfrak L_t:=(V,[\cdot,\cdot]_t)$ is isomorphic to $\mathfrak L$.

\begin{definition}\label{def:ContinuousContraction}
If the limit $[x,y]_0:=\lim_{t\to0^+}[x,y]_t$ exists for any $x,y\in V$,
then $[\cdot,\cdot]_0$ is a well-defined Lie bracket on $V$ and the Lie algebra $\mathfrak L_0:=(V,[\cdot,\cdot]_0)$
is called a {\it one-parametric continuous contraction}
(or simply {\it contraction}) of the Lie algebra $\mathfrak L=(V,[\cdot,\cdot])$.
\end{definition}

Among the functions $U\colon(0,1]\to{\rm GL}(V)$
those realized by specific diagonal matrices are particularly important
for the study of contractions and their applications in mathematical physics.
Consider the following family of such functions:
\begin{gather}\label{eq:GeneralizedIWMatrix}
U_t=\bigoplus_{s=0}^mt^{n_s}{\rm Id}_{V^{(s)}},\quad\mbox{where}\quad
V=\bigoplus_{s=0}^m V^{(s)}
\quad\mbox{and}\quad
n_0,\dots,n_m\in\mathbb Z,\quad m\geqslant1.
\end{gather}
Choosing a basis in each subspace $V^{(s)}$, we represent $U_t$ by a diagonal matrix with entries $t^{n_s}$, each repeated $\dim V^{(s)}$ times.
Then the limit $[\cdot,\cdot]_0$ exists if and only if
the following constraints are satisfied:
\begin{gather}\label{eq:GeneralizedIWConstraints}
n_{s_i}+n_{s_j}\geqslant n_{s_k},\quad
\mbox{whenever}\quad
c_{ij}^k\ne0,\quad
i,j,k=1,\dots,\dim V,
\end{gather}
where $c_{ij}^k$ stands for the components of the structure tensor
of the Lie algebra~$\mathfrak L$ in the chosen basis, with the $i$-th basis element belonging to $V^{(s_i)}$.
When these conditions hold,
the contracted algebra $\mathfrak L_0$ exists 
and is referred as a {\it generalized In\"on\"u--Wigner contraction} of $\mathfrak L$.
In the particular case $m=1$, $n_0=0$ and $n_1=1$,
the contracted algebra $\mathfrak L_0$ is called
a {\it simple In\"on\"u--Wigner contraction}.

\begin{remark}
As noted in~\cite{grun1988a}, we can associate to the decomposition
$V=\bigoplus_{s=0}^mV^{(s)}$ in~\eqref{eq:GeneralizedIWMatrix}
the filtration
\[
\mathfrak L:=\bigcup_{n\in\mathbb Z}\mathfrak L(n)
\quad\mbox{with}\quad
\mathfrak L(n)=\bigoplus_{i\in\{0,\dots,m\},\ n_i\leqslant n}V^{(i)}.
\]
The constraints~\eqref{eq:GeneralizedIWConstraints} are necessary and sufficient to ensure that
$[\mathfrak L(n),\mathfrak L(m)]\subset\mathfrak L(n+m)$.
The associated graded algebra $\operatorname{gr}\mathfrak L$
of $\mathfrak L$ with respect to the above filtration
is isomorphic to the contracted algebra $\mathfrak L_0$
under the contraction matrix $U(\varepsilon)$ from~\eqref{eq:GeneralizedIWMatrix}.
This follows directly from the description of the structure tensor $\tilde c_{ij}^k$ of $\mathfrak L_0$,
see, e.g., \cite[Section~3]{nest2006a}:
\[
\tilde c_{ij}^k=c_{ij}^k
\quad\mbox{if}\quad
n_{s_i}+n_{s_j}=n_{s_k}
\quad\mbox{and}\quad
\tilde c_{ij}^k=0
\quad\mbox{otherwise}.
\]
\end{remark}

\par\noindent
{\it Degenerations}.
Consider an $n$-dimensional vector space $V$ over a field $\mathbb F$.
Denote by $\mathcal L_n(\mathbb F)$ the variety of Lie algebra laws on $V$,
i.e., the set of all alternating bilinear maps from $V\times V$ to $V$ satisfying the Jacobi identity.
The set $\mathcal L_n(\mathbb F)$ is an algebraic set in the affine space $\wedge^2V^*\otimes V$,
which is invariant under the natural action
of ${\rm GL}(V)$ on $\wedge^2V^*\otimes V$,
given by
\begin{gather*}
(g\cdot\mu)(x,y)=g(\mu(g^{-1}x,g^{-1}y))
\quad\mbox{for all}\ \  
g\in{\rm GL}(V),\ 
\mu\in\mathcal L_n(\mathbb F)
\ \mbox{and}\ 
x,y\in V.
\end{gather*}
(Note that this is a left action, whereas the one in~\eqref{eq:ParameterizedLieBracket}, following the tradition in physics literature, is on the right.)
Denote by $\mathcal O(\mu)$ the orbit of $\mu\in\mathcal L_n(\mathbb F)$ under the action of ${\rm GL}(V)$
and by $\overline{\mathcal O(\mu)}$
its closure with respect to the Zariski topology in $\mathcal L_n(\mathbb F)$.

\begin{definition}\label{def:Degenerations}
The Lie algebra $\mathfrak L_0:=(V,\mu_0)$ is called a {\it degeneration}
of the Lie algebra $\mathfrak L:=(V,\mu)$ if $\mu_0\in\overline{\mathcal O(\mu)}$.
\end{definition}

When $\mathbb F$ is algebraically closed,
the orbit $\mathcal O(\mu)$ is open in its Zariski closure $\overline{\mathcal O(\mu)}$
by the {\it closed orbit lemma} in~\cite{bore1991A}.
Hence, the boundary $\overline{\mathcal O(\mu)}\setminus \mathcal O(\mu)$ is a union of orbits of strictly lower dimension.
Moreover, all points of $\mathcal O(\mu)$ are regular in $\overline{\mathcal O(\mu)}$.
Finally, the orbit $\mathcal O(\mu)$ is a constructible set,
so in the case $\mathbb F=\mathbb C$, the closure of $\mathcal O(\mu)$ in Zariski topology
coincides with its closure in the metric topology,
see~\cite[Corollary~1, p.~60]{mumf1999A}.
Therefore, over $\mathbb C$, the concepts of contraction and degeneration coincide.

It is worth pointing out that, over $\mathbb R$, any contraction is a degeneration.
The converse probably fails because the metric closure of an orbit may be strictly contained
in the Zariski closure, see examples in~\cite{eber2009a}.

\subsection{Gradings and graded contractions}\label{sec:GradingsAndContractions}

Let $G$ be an abelian group and $\mathbb F$ be a field.
A {\it $G$-grading} on an arbitrary $\mathbb F$-algebra $A$ (not necessarily unital or associative)
is a vector space decomposition $\Gamma\colon A=\bigoplus_{g\in G}A_g$
such that $A_gA_h\subseteq A_{gh}$ for all $g,h\in G$.
The elements of $A_g$ are said to be {\it homogeneous of degree $g$}
and the subspaces $A_g$ are called {\it homogeneous components}.
The {\it support} of the grading $\Gamma$
is the set ${\rm Supp}(\Gamma):=\{g\in G\colon A_g\ne0\}$.
If $\Gamma$ is fixed, $(A,\Gamma)$ is said to be a \emph{$G$-graded algebra}.
We may write~$A$ for a graded algebra if the grading is clear from the context.
For more detailed exposition see, e.g.,~\cite{eldu2013A}.

\medskip\par\noindent
{\it Graded contractions}.
Denote the multiplication of an algebra~$A$ by~$\mu$.
Given an arbitrary map $\varepsilon\colon G\times G\to\mathbb F$, we can define a new multiplication on the underlying vector space of $A$: 
\[
\mu^\varepsilon(x,y)=\varepsilon(g,h)\mu(x,y)
\quad\mbox{for all}\ \ 
g,h\in G,\ 
x\in A_g,\
y\in A_h,
\]
which defines a new algebra, $A^\varepsilon$.
However, if $A$ belongs to a specific variety (e.g., associative, commutative, Lie or Jordan algebras),
$A^\varepsilon$ is not guaranteed to belong to the same variety unless we impose restrictions on $\varepsilon$.
For example, if $\varepsilon$ is a two-cocycle on $G$ with values in $\mathbb F^\times$,
then~$A^\varepsilon$ is known as a \emph{cocycle twist} of $A$. In this case, if $A$ is an associative algebra,
then so is $A^\varepsilon$; also $(A^\varepsilon)^{\varepsilon^{-1}}=A$.
In the context of Lie algebras, the definition of a graded contraction takes the following form:

\begin{definition}\label{def:GradContLieAlg}
Let $\Gamma\colon\mathfrak L=\bigoplus_{g\in G}\mathfrak L_g$ be a $G$-grading on a Lie algebra~$\mathfrak L$.
A \emph{graded contraction} of $(\mathfrak L,\Gamma)$ is a map $\varepsilon\colon G\times G\to\mathbb F$
such that $\mathfrak L^\varepsilon$ is a Lie algebra.
Then $\mathfrak L^\varepsilon$ is called the {\it contracted algebra} or, by abuse of language, the \emph{graded contraction} of $(\mathfrak L,\Gamma)$ with respect to $\varepsilon$.
\end{definition}

\noindent
The algebra $\mathfrak L^\varepsilon$ is a Lie algebra if and only if the following conditions are satisfied:
\begin{enumerate}\itemsep=0ex
\item[$(i)$]
$\big(\varepsilon(g,h)-\varepsilon(h,g)\big)[x,y]=0$,

\item[$(ii)$]
$\big(\varepsilon(g,hk)\varepsilon(h,k)-\varepsilon(k,gh)\varepsilon(g,h)\big)[x,[y,z]]
+\big(\varepsilon(h,kg)\varepsilon(k,g)-\varepsilon(k,gh)\varepsilon(g,h)\big)[y,[z,x]]=0$
\end{enumerate}
for all $x\in\mathfrak L_g$, $y\in\mathfrak L_h$ and $z\in\mathfrak L_k$.
Conditions $(i)$ and $(ii)$ arise from the antisymmetry 
and the Jacobi identity, respectively,  applied to $\mathfrak L^\varepsilon$.

Fixing a Lie algebra~$\mathfrak L$ and a $G$-grading~$\Gamma$ on it,
the investigation of graded contractions of~$(\mathfrak L,\Gamma)$ can be quite laborious \cite{hriv2006a,hriv2013a}.
Trivially,  when~$\mathfrak L_{gh}=0$ and~$\mathfrak L_{ghk}=0$, 
conditions~$(i)$ and~$(ii)$ are satisfied automatically,
imposing no constraints on the values~$\varepsilon(g,h)$.

To develop a general theory of graded contractions,
it is natural to disregard the specifics of the grading on $\mathfrak L$
and consider the so-called ``generic case'', i.e.,
require that $(i)$ and $(ii)$ hold {\it for all} $G$-graded Lie algebras.

\begin{proposition}\label{prop:GenericContraction}
The necessary and sufficient conditions for the map $\varepsilon\colon G\times G\to\mathbb F$
to define a $G$-graded contraction for all $G$-graded Lie algebras are
\begin{enumerate}\itemsep=0ex
\item[$(i')$]
$\varepsilon(g,h)=\varepsilon(h,g)$,
	
\item[$(ii')$]
$\varepsilon(g,hk)\varepsilon(h,k)=\varepsilon(k,gh)\varepsilon(g,h)$
\end{enumerate}
for all $g,h,k\in G$.
\end{proposition}

\begin{proof}
Since $(ii')$ implies $\varepsilon(k,gh)\varepsilon(g,h)
=\varepsilon(h,kg)\varepsilon(k,g)$ by renaming variables, it is clear that $(i')$ and $(ii')$ are sufficient
for $\varepsilon$ to define a graded contraction for any $(\mathfrak L,\Gamma)$.

To prove the necessity, it suffices to find a $G$-graded Lie algebra $\mathfrak L$ satisfying
$(i)$ $[\mathfrak L_g,\mathfrak L_h]\ne0$ for all distinct $g,h\in G$,
$(ii)$ for all $g,h,k\in G$ there exist $x\in\mathfrak L_g$, $y\in\mathfrak L_h$, $z\in\mathfrak L_k$
such that $[x,[y,z]]$ and $[y,[z,x]]$ are linearly independent.
For instance, consider the free Lie algebra $\mathfrak F$ on the generators $x_g$, $y_g$ and $z_g$,
where~$g$ runs through $G$.
This algebra has a natural grading by setting the $G$-degree of
generators $x_g$, $y_g$ and $z_g$ to be $g$.
Since $[x_g,x_h]\ne0$ for all $g\ne h$, 
condition $(i)$ implies $(i')$ for the graded Lie algebra $\mathfrak F$. 
Since the commutators $[x_g,[y_h,z_k]]$ and $[y_h,[z_k,x_g]]$ are linearly independent in $\mathfrak F$, condition $(ii)$ implies $(ii')$. 
\end{proof}

\begin{remark} 
If $G$ is finite, then we can restrict to finite-dimensional Lie algebras in the above proposition. Indeed, we can replace the Lie algebra $\mathfrak F$ with the finite-dimensional nilpotent Lie algebra $\mathfrak F/\mathfrak F^4$, where $\mathfrak F^4$ is the fourth term of the lower central series of $\mathfrak F$. 
\end{remark}

When $\varepsilon$ does not take value $0$, conditions $(i')$ and $(ii')$ say that $\varepsilon$ is a symmetric two-cocycle with values in $\mathbb F^\times$ with trivial $G$-action: $\varepsilon\in {\rm Z}^2_{\rm sym}(G,\mathbb F^\times)$, see Section~\ref{sec:PrelimGroupCohomology}.

In what follows, we restrict our attention to this generic point of view,
as was done, for example, in \cite{mood1991a,weim2006a}.
All definitions and examples in this paper are presented within this framework.
In particular, when using the notion ``graded contraction'' we omit referring to 
a $G$-graded Lie algebra~$(\mathfrak L,\Gamma)$.
For a ``grading-dependent'' treatment of contractions, see, for example, \cite{drap2024b,drap2024a,mont1991a}.

The obvious examples of the graded contractions are the constant maps $\varepsilon=1$ and $\varepsilon=0$.
Their corresponding contracted algebras are respectively the original algebra $\mathfrak L$ and the abelian Lie algebra.
Both these maps are examples of {\it trivial graded contractions}, that is,
the graded contractions $\varepsilon\colon G\times G\to\mathbb F$ such that the for all $G$-graded Lie algebras $\mathfrak L$
the contracted algebra~$\mathfrak L^\varepsilon$  is isomorphic to~$\mathfrak L$
or to the abelian Lie algebra on the vector space $\mathfrak L$, see~\cite[Definition~2.6]{weim2006a}.

A more interesting example of a trivial graded contraction is that arising
from the normalization of homogeneous components of the grading $\Gamma$.
Let $\alpha\colon G\to\mathbb F^\times$ be an arbitrary map,
then the map $d\alpha\colon G\times G\to\mathbb F^\times$ defined in Section~\ref{sec:PrelimGroupCohomology}
is a graded contraction of $\mathfrak L$.
It is clear that $\mathfrak L^{d\alpha}$ is isomorphic to $\mathfrak L$.
The choice of the notation $d\alpha$ for this kind of contractions
is justified by the fact that it is the coboundary of the map $\alpha$
if we consider $\mathbb F^\times$ as a $G$-module with trivial action.
Furthermore, when $\alpha$ is a group homomorphism then $d\alpha=1$.

\medskip\par\noindent
{\it Graded degenerations}.
Let $G$ be an abelian group. Consider a $G$-grading $\Gamma$ on a finite-dimensional vector space $V$ over an algebraically closed field $\mathbb F$, $V=\bigoplus_{g\in G}V_g$, 
with $\dim V_g=n_g\in\mathbb N\cup\{0\}$.
Denote $\underline n=(n_g)_{g\in G}$ and let $\mathcal L_{\underline n}(\mathbb F)$ be the variety of Lie algebra laws on $V$ that make it a $G$-graded algebra.
The set $\mathcal L_{\underline n}(\mathbb F)$ is an algebraic set in the affine space $\wedge^2V^*\otimes V$,
which is invariant under the natural action
of the subgroup $\prod_{g\in G}{\rm GL}(V_g)$ of ${\rm GL}(V)$. 
We will consider the subgroup $D(\Gamma)$ of $\prod_{g\in G}{\rm GL}(V_g)$ consisting of the operators that act as (nonzero) scalars on each $V_g$.
It is clear that $D(\Gamma)$ is an algebraic torus of dimension $|{\rm Supp}(\Gamma)|$.

We can define graded degenerations in a way similar to how ordinary degenerations
were defined in Section~\ref{sec:ContAndDegenOfLieAlgs}, just replacing ${\rm GL}(V)$ with the group $D(\Gamma)$.
It turns out that the graded degenerations $\mu_0$ of a Lie algebra law $\mu\in\mathcal L_{\bar n}(\mathbb F)$
depend on the set $S={\rm Supp}_2(\Gamma)$ defined by
\[
{\rm Supp}_2(\Gamma):=\{(g,h)\in G\times G\mid \mu(V_g,V_h)\ne0\},
\]
which we will call the {\it second support} of the grading $\Gamma$.
Note that, in the case of Lie algebras, the order of $g$ and $h$ does not matter,
so ${\rm Supp}_2(\Gamma)$ can be regarded as a subset of the set $\mathcal P_G$ of unordered pairs $\{g,h\}$.
We discuss this in detail in Section~\ref{sec:VarietyOfContractions}.

\begin{proposition}\label{prop:SecondSupport}
Any subset of $\mathcal P_G$ is the second support of a $G$-graded Lie algebra.
Moreover, if $G$ is finite, this Lie algebra can be chosen finite-dimensional.
\end{proposition}

\begin{proof}
Consider the free Lie algebra $\mathfrak F$ on the generators $x_g$ and $y_g$,
where~$g$ runs through $G$.
We can define a $G$-grading on $\mathfrak F$
by setting the $G$-degree of generators $x_g$ and $y_g$ to be $g$.
Given a set $S\subset\mathcal P_G$,
consider the ideal $I_S$ of $\mathfrak F$ generated by $\mathfrak F^3$ (the third term of the lower central series of $\mathfrak F$)
and the commutators $[\mathfrak F_g,\mathfrak F_h]$ for all $\{g,h\}\notin S$.
Then the quotient $\mathfrak L:=\mathfrak F/I_S$ is a $G$-graded Lie algebra,
whose second support is~$S$.
Moreover, if $G$ is finite, the algebra $\mathfrak L$ is finite-dimensional.
\end{proof}

\begin{remark}
Replacing ${\rm GL}(V)$ with the product $\prod_{g\in G}{\rm GL}(V_g)$,
one can define the concept of {\it degeneration of a $G$-graded Lie algebra},
which is different from graded degeneration defined above.
\end{remark}

\medskip\par\noindent
{\it Equivalences of graded contractions}.
Several natural equivalence relations can be introduced on the set of graded contractions.
For the generic setting, some of them already appeared in~\cite{mood1991a} and~\cite{weim1995a}.
There are also versions of these equivalences for a specific $G$-graded Lie algebra $\mathfrak L$.

Recall that, for $G$-graded algebras $A$ and $B$,
an algebra isomorphism $\varphi\colon A\to B$ is an {\it equivalence} if it maps each homogeneous component of $A$ 
onto a homogeneous component of $B$: for any $g\in G$ there is $h\in G$ such that
$\varphi\left(A_g\right)=B_h$,
and $\varphi$ is a {\it graded isomorphism} if $\varphi\left(A_g\right)=B_g$ for all $g\in G$.

\begin{definition}[\cite{drap2024a}]\label{def:EquivalentContractions}
For a particular $G$-graded Lie algebra $\mathfrak L$,
two graded contractions $\varepsilon$ and $\varepsilon'$
are called {\it equivalent} (resp. {\it strongly equivalent})
if the contracted Lie algebras $\mathfrak L^\varepsilon$ and $\mathfrak L^{\varepsilon'}$
are equivalent (resp. graded isomorphic).
\end{definition}

The following is an analogue of strong equivalence for generic graded contractions:

\begin{definition}\label{def:StrongEquiv}
Two contractions $\varepsilon$ and $\varepsilon'$ are called {\it isomorphic},
written $\varepsilon\simeq\varepsilon'$,
if, for any $G$-graded Lie algebra~$\mathfrak L$,
the contracted Lie algebras $\mathfrak L^\varepsilon$ and $\mathfrak L^{\varepsilon'}$
are graded isomorphic.
\end{definition}

For arbitrary graded contractions $\varepsilon$ and $\varepsilon'$
one can define their product $\varepsilon\varepsilon'\colon G\times G\to\mathbb F$
to be simply the point-wise product of functions,
\[
(\varepsilon\varepsilon')(g,h):=\varepsilon(g,h)\varepsilon'(g,h).
\]
It is clear that~$\varepsilon\varepsilon'$ is also a graded contraction.
In particular, if $\varepsilon$ is a graded contraction,
then $\varepsilon\, d\alpha$ is a graded contraction for any function $\alpha\colon G\to\mathbb F^\times$.
This gives rise to the third notion of equivalence of graded contractions.

\begin{definition}[\cite{drap2024a}]\label{def:EquivViaNormalization}
Two graded contractions $\varepsilon$ and $\varepsilon'$ are called {\it equivalent via normalization},
written $\varepsilon\sim_n\varepsilon'$,
if there exists a map $\alpha\colon G\to\mathbb F^\times$ such that
$\varepsilon'=\varepsilon\, d\alpha$.
\end{definition}

Definition~\ref{def:EquivViaNormalization} can be rephrased as follows:
two maps $\varepsilon,\varepsilon'\colon G\times G\to\mathbb F$
satisfying the ``two-cocycle'' condition $(ii')$ of Proposition~\ref{prop:GenericContraction} are equivalent via normalization
if and only if they are cohomologous.
This provides us with the cohomological framework for the study of graded contractions,
which is discussed in more detail in Section~\ref{sec:GroupCohomology}.

The equivalence via normalization implies isomorphism,
$\varepsilon\sim_n\varepsilon'\Rightarrow\varepsilon\simeq\varepsilon'$.
The graded isomorphism $\varphi\colon\mathfrak L^{\varepsilon}\to\mathfrak L^{\varepsilon'}$
is given by
\[
\varphi(x)=\alpha(g)x\quad\forall x\in\mathfrak L_g
\] 
for all $G$-graded Lie algebras $\mathfrak L$.
The converse implication, $\varepsilon\simeq\varepsilon'\Rightarrow\varepsilon\sim_n\varepsilon'$, is conjectured in~\cite[Conjecture~2.15]{weim2006a}.
While $\mathfrak L^\varepsilon$ may be isomorphic to $\mathfrak L^{\varepsilon'}$, for specific $\mathfrak L$,
without $\varepsilon$ and $\varepsilon'$ being equivalent via normalization (see, for example,~\cite[Example~2.12]{drap2024a}),
we are not aware of counterexamples to this conjecture. 
Moreover, we show that a version of this conjecture holds
when the graded isomorphism for all $\mathfrak L$ is replaced by its functorial counterpart, i.e., 
{\it natural graded isomorphism}.
This is discussed in detail in Section~\ref{sec:EquivOfContractions}.

\section{Graded contractions via group cohomology}\label{sec:GroupCohomology}

In~\cite[Section~3]{mood1991a} Moody and Patera initiated the study of ``generic'' $G$-graded Lie algebra contractions,
which they defined as functions $\varepsilon\colon G\times G\to\mathbb F$
satisfying conditions $(i')$ and $(ii')$ of Proposition~\ref{prop:GenericContraction},
from the perspective of semigroup cohomology.
Indeed, since the field $\mathbb F$ is a semigroup with respect to multiplication,
it can be regarded as a $G$-semimodule with trivial action.
In view of conditions~$(i')$ and~$(ii')$,
a contraction~$\varepsilon$ is a symmetric two-cocycle with values in this $G$-semimodule,
moreover the equivalence via normalization
coincides with the relation of ``being cohomologous''.
Since cohomology theory is more developed for groups than semigroups,
we will follow a group-cohomology approach, which was inspired by the work of Weimar--Woods~\cite{weim2006a},
even though group cohomology does not appear there explicitly.

\begin{definition}\label{def:Support}
Let $G$ be an abelian group. We denote by $\mathcal P_G$ the set of unordered pairs of elements of~$G$, 
\[
\mathcal P_G:=\{\{g,h\}\subset G\}.
\]
If the group $G$ is finite,
the cardinality of $\mathcal P_G$ equals $|G|(|G|+1)/2$.
In view of condition $(i')$ of Proposition~\ref{prop:GenericContraction},
a $G$-graded contraction $\varepsilon$
can be considered as a function $\mathcal  P_G\to\mathbb F$.
In particular, the {\it support} of~$\varepsilon$ is the set
\[
S(\varepsilon):=\{\{g,h\}\in\mathcal P_G\mid \varepsilon(g,h)\ne0\}.
\]
The contraction $\varepsilon$ is said to be {\it without zeros} if $S(\varepsilon)=\mathcal P_G$,
and {\it with zeros} otherwise.

Conditions $(i')$ and $(ii')$ tell us that
a set $S\subset\mathcal P_G$ is a support of a contraction if and only if
$\{g,h\},\{gh,k\}\in S$ implies $\{h,k\},\{g,hk\}\in S$ for all $g,h,k\in G$.
We denote the set of all such subsets of $\mathcal P_G$ by $\mathcal S(G)$.
\end{definition}

Finding supports is an essential step in the classification of $G$-graded contractions.
While testing if a given subset $S$ of $\mathcal P_G$ belongs to $\mathcal S(G)$ is straightforward for 
finite $G$, finding all elements of the set $\mathcal S(G)$ is a challenging computational problem because of 
the growth rate of the size of the power set of $\mathcal P_G$ as a function of $N:=|G|$.

To find the set $\mathcal S(G)$, we develop an
approach based on some elementary facts from commutative algebra
and affine algebraic geometry.
It relies on the efficient algorithms of computing Gr\"obner bases and taking quotient rings
in the computer algebra system {\sc Singular}~\cite{deck2024A}.
We discuss this algorithm in detail in Section~\ref{sec:VarietyOfContractions}.

\begin{remark}\label{rem:MeetSemilattice}
The set $\mathcal S(G)$ is closed with respect to taking intersections,
i.e., if $S_1,S_2\in\mathcal S(G)$, then the intersection $S_1\cap S_2$ also belongs to $\mathcal S(G)$.
Thus, $\mathcal S(G)$ is a meet semilattice with respect to inclusion.
This observation suggests addressing the problem from the perspective of the semilattice structure of $\mathcal S(G)$.
Indeed, given all elements of $\mathcal S(G)$ that are maximal with the property of not containing a fixed element of $\mathcal P_G$, 
all other elements of $\mathcal S_G$ can be generated by taking all possible intersections.
This approach avoids solving polynomial equations that define contractions.
\end{remark}

Let $\mathcal F_G$ be the free abelian group, written multiplicatively, on the alphabet $\mathcal P_G$,
and let $\mathcal R_G$ be the set of relators of the form
\[
\{g,h\}\{gh,k\}\{h,k\}^{-1}\{g,hk\}^{-1},
\]
for all $g,h,k\in G$.
Any contraction with support $\mathcal P_G$ (i.e., without zeros) defines a unique homomorphism $\mathcal F_G\to\mathbb F^\times$.
Under this homomorphism, the elements of the subgroup $\langle\mathcal R_G\rangle$ are mapped to $1\in\mathbb F$,
so we will call them the {\it higher-order identities}.
Thus, the group of contractions with support $\mathcal P_G$
may be identified with the group of multiplicative characters of the group $\mathcal F_G/\langle\mathcal R_G\rangle$,
${\rm Hom}_{\mathbb Z}(\mathcal F_G/\langle\mathcal R_G\rangle,\mathbb F^\times)$.
In what follows, we simply write ${\rm Hom}(A,B)$
for the group of abelian group homomorphisms $A\to B$.

Since we are interested in contractions with zeros,
it is important to note that every relation $r=1$ for $r\in\langle\mathcal R_G\rangle$
can be uniquely rewritten in the form $r_1=r_2$,
where $r_1$ and $r_2$ are monomials in two disjoint sets of generators (without using inverses).
We note that the degrees of these two monomials are necessarily equal,
because the homomorphism $\mathcal F_G\to\mathbb Z$ that maps each generator of $\mathcal F_G$ to $1$
sends all elements of $\mathcal R_G$ and hence the subgroup $\langle\mathcal R_G\rangle$ to zero.

The following list of definitions is an interpretation of definitions from~\cite{weim2006a}
within the framework introduced above.

\begin{definition}\label{def:FreeGroupsFramework}
For any subset $S\subset\mathcal P_G$ let $\mathcal F_S$
be the free abelian group generated by $S$.
\begin{enumerate}\itemsep=0ex
\item 
The elements of the set $\mathcal R_S:=\mathcal R_G\cap\mathcal F_S$ are called {\it surviving defining relations}.
The elements of the subgroup $\langle\mathcal R_G\rangle\cap\mathcal F_S$ are called {\it surviving higher-order identities}.
 
\item 
A subset $B\subset\mathcal P_G$ is called {\it independent}
if its image under the projection
$\mathcal F_G\twoheadrightarrow\mathcal F_G/\langle\mathcal R_G\rangle$
is $\mathbb Z$-linearly independent.
		
\item 
A subset $B\subset S$ is called {\it independent} 
(resp. {\it quasi-independent}) {\it with respect to~$S$}
if the image of $B$ under the projection
$\mathcal F_S\twoheadrightarrow\mathcal F_S/(\langle\mathcal R_G\rangle\cap\mathcal F_S)$
(resp. $\mathcal F_S\twoheadrightarrow\mathcal F_S/\langle\mathcal R_S\rangle$)
is $\mathbb Z$-linearly independent.
		
\item 
We denote by $N'$ (resp. $N''$) the free rank of the group
$\mathcal F_S/(\langle\mathcal R_G\rangle\cap\mathcal F_S)$
(resp. $\mathcal F_S/\langle\mathcal R_S\rangle$),
i.e., the cardinality of a maximal subset of~$S$
that is independent (resp. quasi-independent) with respect to~$S$.
\end{enumerate}
\end{definition}

\begin{remark}
In~\cite{weim2006a}, only supports of contractions, i.e., $S\in\mathcal S(G)$, were considered,
but the definitions make sense for any subset of $\mathcal P_G$. 
Also, in~\cite{weim2006a}, the relations $r_1=r_2$ of degrees $1$ and $2$
were excluded from higher-order identities
because, for $S\in\mathcal S(G)$, they belong to $\langle\mathcal R_S\rangle$,
i.e., follow from surviving defining relations.
Indeed, a surviving identity of degree $1$ has the form $\{g_1,h_1\}=\{g_2,h_2\}$,
where both sides are distinct and belong to~$S$.
Consider the group homomorphism $f\colon\mathcal F_G\to\mathbb ZG$ defined by $\{g,h\}\mapsto g+h-gh$.
Since~$f$ sends $\mathcal R_G$ to zero, it must kill all elements of $\langle\mathcal R_G\rangle$,
in particular $\{g_1,h_1\}\{g_2,h_2\}^{-1}$, so we have $g_1+h_1-g_1h_1=g_2+h_2-g_2h_2$.
This implies that the degree~$1$ identity has the form $\{e,e\}=\{e,g\}$ with both sides belonging to $S$,
so it follows from the surviving defining relation $\{e,e\}\{e,g\}=\{e,g\}\{e,g\}$.
Similar considerations work for degree $2$ showing
that the surviving degree $2$ identities are in fact surviving defining relations.


\end{remark}

The inclusion $\langle\mathcal R_S\rangle \hookrightarrow\langle\mathcal R_G\rangle\cap\mathcal F_S$
induces a surjection $\mathcal F_S/\langle \mathcal R_S\rangle\twoheadrightarrow\mathcal F_S/(\langle\mathcal R_G\rangle\cap\mathcal F_S)$.
This gives us that $N''\geqslant N'$.
Since the free rank of an abelian group $A$ is the dimension of the $\mathbb Q$-vector space $A\otimes_{\mathbb Z}\mathbb Q$, 
the short exact sequence of abelian groups
\begin{gather*}
0\rightarrow
(\langle\mathcal R_G\rangle\cap \mathcal F_S)/\langle\mathcal R_S\rangle\rightarrow
\mathcal F_S/\mathcal \langle\mathcal R_S\rangle\rightarrow
\mathcal F_S/(\mathcal \langle\mathcal R_G\rangle\cap\mathcal F_S)\rightarrow
0
\end{gather*}
implies the following equality:
\[
\mathop{\rm rank}\big(\mathcal F_S/\langle\mathcal R_S\rangle\big)
=\mathop{\rm rank}\big(\mathcal F_S/(\langle\mathcal R_G\rangle\cap\mathcal F_S)\big)
+\mathop{\rm rank}\big((\langle\mathcal R_G\rangle\cap\mathcal F_S)/\langle\mathcal R_S\rangle\big),
\]
hence $\mathop{\rm rank}\big((\langle\mathcal R_G\rangle\cap\mathcal F_S)/\langle\mathcal R_S\rangle\big)=N''-N'$ 
if $G$ is finite.

For any abelian group $A$, viewed as a trivial $G$-module,
the set of symmetric two-cocycles ${\rm Z}^2_{\rm sym}(G,A)$ 
with values in $A$ can be identified with ${\rm Hom}(\mathcal F_G/\langle\mathcal R_G\rangle,A)$.
It is well-known that 
\begin{equation}\label{eq:Ext}
{\rm Z}^2_{\rm sym}(G,A)/{\rm B}^2(G,A)\simeq {\rm Ext}(G,A),
\end{equation}
see Section~\ref{sec:PrelimGroupCohomology} for more details and references.
The group ${\rm Ext}(G,A)$ can be computed using the cochain complex ${\rm Hom}(C_{\bullet},A)$
for any free resolution $C_{\bullet}$ of the abelian group~$G$.

Consider the following complex of abelian groups:
\begin{gather*}
C_{\bullet}\colon\quad
C_3\to C_2\to C_1\to C_0\to0,
\end{gather*}
where $C_0:=G$, 
$C_1:=\mathbb ZG$,
$C_2:=\mathcal F_G$
and $C_3:=\mathbb Z(G\times G\times G)$
and the differential $\p_\bullet$ is defined as follows:
\begin{gather*}
\p_1\colon C_1\to C_0,\quad
\sum_{g\in G}n_gg\mapsto\prod_{g\in G}g^{n_g},
\\
\p_2\colon C_2\to C_1,\quad
\{g,h\}\mapsto g+h-gh,
\\
\p_3\colon C_3\to C_2,\quad
(g,h,k)\mapsto\{h,k\}\{g,hk\}\{g,h\}^{-1}\{gh,k\}^{-1}. 
\end{gather*}

\begin{lemma}\label{lem:CisExact}
The sequence $C_{\bullet}$ is exact, i.e., a free resolution of $C_0=G$.
\end{lemma}

\begin{proof}
The exactness at the terms $C_0$ and $C_1$ is clear.
To show exactness at $C_2$,
we can apply the exact faithful functor ${\rm Hom}(\cdot,\mathbb Q/\mathbb Z)$,
see, e.g., \cite[Chapter 2, Exercise 2.5.1]{weib1994A},
and see that the resulting sequence is exact at ${\rm Hom}(C_2,\mathbb Q/\mathbb Z)$,
i.e., ${\rm Z}^2_{\rm sym}(G,\mathbb Q/\mathbb Z)={\rm B}^2(G,\mathbb Q/\mathbb Z)$,
because, in view of~\eqref{eq:Ext}, the quotient ${\rm Z}^2_{\rm sym}(G,\mathbb Q/\mathbb Z)/{\rm B}^2(G,\mathbb Q/\mathbb Z)$
is ${\rm Ext}(G,\mathbb Q/\mathbb Z)$,
which is trivial for any~$G$ since $\mathbb Q/\mathbb Z$ is a divisible group.
\end{proof}

\begin{corollary}\label{cor:F_G/R_GIsFree}
$\mathcal F_G/\langle\mathcal R_G\rangle$ is a free abelian group of rank $\leqslant |G|$.
If $N:=|G|$ is finite then $\mathcal F_G/\langle\mathcal R_G\rangle$ has rank $N$.
\end{corollary}

\begin{proof}
By Lemma~\ref{lem:CisExact}, $\langle\mathcal R_G\rangle=\mathop{\rm im}\p_3=\ker\p_2$,
so $\p_2$ induces an embedding $\mathcal F_G/\langle\mathcal R_G\rangle\hookrightarrow C_1$.
Since $C_1$ is a free abelian group of rank $N$, it follows that $\mathcal F_G/\langle\mathcal R_G\rangle$ is 
a free abelian group of rank $\leqslant N$.
If $G$ is finite, then the rank of $\mathcal F_G/\langle\mathcal R_G\rangle$ is equal to $N$,
since the cokernel of $\p_2$ is finite (isomorphic to~$G$).
\end{proof}

Now consider a subset $S\subset \mathcal P_G$
and the corresponding groups $\mathcal F_S$, $\mathcal R_S$ and $\langle\mathcal R_G\rangle\cap\mathcal F_S$
as in Definition~\ref{def:FreeGroupsFramework}.
Recall that the free ranks of the quotient groups
$\mathcal F_S/\langle\mathcal R_S\rangle$ and $\mathcal F_S/(\langle\mathcal R_G\rangle\cap\mathcal F_S)$
are denoted by $N''$ and $N'$, respectively.
Since the group $\mathcal F_S/(\langle\mathcal R_G\rangle\cap\mathcal F_S)$
is isomorphic to the image of $\mathcal F_S$ under the quotient map $\mathcal F_G\to\mathcal F_G/\langle\mathcal R_G\rangle$,
it is a free abelian group of rank $N'\leqslant N$.

\begin{remark}\label{eq:F_S/R_SIsFree}
It is claimed in~\cite[Theorem~6.5]{weim2006a} that
if $S$ belongs to $\mathcal S(G)$ then the group $(\langle\mathcal R_G\rangle\cap\mathcal F_S)/\langle\mathcal R_S\rangle$
is also free (of rank $Q=N''-N'$), but the proof appears to be incorrect
(in the last paragraph of the proof, a function taking values in roots of unity
is treated as a constant).
Since we know that $\mathcal F_S/(\langle\mathcal R_G\rangle\cap\mathcal F_S)$ is free,
the freeness  of $(\langle\mathcal R_G\rangle\cap\mathcal F_S)/\langle\mathcal R_S\rangle$ is equivalent
to the freeness of $\mathcal F_S/\langle\mathcal R_S\rangle$.
We carried out numerical experiments for small groups~$G$ and verified that the group $\mathcal F_S/\langle\mathcal R_S\rangle$
is free not only for $S\in\mathcal S(G)$, but for all subsets $S$ of $\mathcal P_G$.
\end{remark}

For $S\in\mathcal S(G)$, the significance of the group $\mathcal Z_S:=\mathcal F_S/\langle\mathcal R_S\rangle$
is that the group of multiplicative characters ${\rm Hom}(\mathcal Z_S,\mathbb F^\times)$  
is isomorphic to the group of contractions whose support is $S$.
Indeed, given $\gamma\in{\rm Hom}(\mathcal Z_S,\mathbb F^\times)$,
we define $\varepsilon\colon G\times G\to\mathbb F$ by
\[
\varepsilon(g,h):=
\begin{cases}
\gamma(\{g,h\})&\mbox{if}\ \{g,h\}\in S,\\
0              &\mbox{otherwise}.
\end{cases}
\]
Condition $(i')$ of Proposition~\ref{prop:GenericContraction} is obviously satisfied.
Relations $\mathcal R_S$ imply that condition $(ii')$ is satisfied when both sides are nonzero
(there are no ``strong violations'' in the terminology of~\cite{weim2006a}),
while the fact that $S\in\mathcal S(G)$ implies that the sides of $(ii')$
are either both zero or both nonzero
(there are no ``weak violations'' in the terminology of~\cite{weim2006a}).
Conversely, every graded contraction with support $S$ has the above form.

The group of equivalence classes of these contractions is isomorphic
to the quotient by the image of ${\rm B}^2(G,\mathbb F^\times)$ under the restriction map
${\rm Fun}(\mathcal P_G,\mathbb F^\times)
\to{\rm Fun}(S,\mathbb F^\times)$,
where ${\rm Fun}(X,Y)$ denotes the set of functions with domain $X$ and codomain $Y$,
which becomes a group under the point-wise operation whenever $Y$ is a group.
Since $\mathcal F_S$ is the free abelian group generated by $S$,
${\rm Fun}(S,\mathbb F^\times)$ can be identified with the group ${\rm Hom}(\mathcal F_S,\mathbb F^\times)$.
For any abelian group $A$, we denote by ${\rm B}^2_S(A)$ the image of ${\rm B}^2(G,A)$ in 
${\rm Hom}(\mathcal F_S,A)$, which lies in ${\rm Hom}(\mathcal Z_S,A)$.
We also denote
\[
{\rm H}^2_S(A):={\rm Hom}(\mathcal Z_S,A)/{\rm B}^2_S(A).
\]
Then, for $S\in\mathcal S(G)$, the group
${\rm H}^2_S(\mathbb F^\times)={\rm Hom}(\mathcal Z_S,\mathbb F^\times)/{\rm B}^2_S(\mathbb F^\times)$
classifies contractions with support $S$ up to equivalence via normalization.

We want to compare the groups $\mathcal Z_S$ and $\mathcal Z_G:=\mathcal F_G/\langle\mathcal R_G\rangle$.
Since the embedding $\mathcal F_S\hookrightarrow\mathcal F_G$ maps $\langle\mathcal R_S\rangle$ into $\langle\mathcal R_G\rangle$,
there exists a unique group homomorphism $f\colon\mathcal Z_S\to\mathcal Z_G$
making the following diagram commutative:
\begin{center}
\begin{tikzcd}
0\arrow[r]
&\langle\mathcal R_S\rangle\arrow[r]\arrow[d,hook]
&\mathcal F_S\arrow[r]\arrow[d,hook]
&\mathcal Z_S\arrow[r]\arrow[d,"f"]
&0
\\
0\arrow[r]
&\langle\mathcal R_G\rangle\arrow[r]
&\mathcal F_G\arrow[r]
&\mathcal Z_G\arrow[r]
&0 
\end{tikzcd}
\end{center}
Denote $\mathcal K_S:=\ker f$,
$\mathcal I_S:=\mathop{\rm im}f$ and 
$\mathcal C_S:=\mathop{\rm coker}f$.
Then
\begin{gather*}
\mathcal K_S=(\langle\mathcal R_G\rangle\cap\mathcal F_S)/\langle\mathcal R_S\rangle,
\\
\mathcal I_S=(\langle\mathcal R_G\rangle\mathcal F_S)/\langle\mathcal R_G\rangle
\simeq\mathcal F_S/(\langle\mathcal R_G\rangle\cap\mathcal F_S),
\\
\mathcal C_S\simeq\mathcal F_G/(\langle\mathcal R_G\rangle\mathcal F_S).
\end{gather*}

Recall that when the group $G$ has finite order $N$,
the group $\mathcal F_G$ is free abelian of rank $N(N+1)/2$.
The groups $\mathcal F_S$, $\langle\mathcal R_G\rangle$ and $\langle\mathcal R_S\rangle$
 are subgroups of $\mathcal F_G$
defined by explicit sets of generators.
Thus the groups $\langle\mathcal R_G\rangle\cap\mathcal F_S$, $\mathcal K_S$, $\mathcal I_S$ and $\mathcal C_S$
can be computed using the Smith canonical form of an integer matrix.

\begin{theorem}\label{thm:ClassificationOfContractions}
For any $S\subset\mathcal P_G$ and any abelian group $A$, we have two short exact sequences:
\begin{align}
\label{eq:SequenceH^2_S} 
0
\longrightarrow &\frac{{\rm Hom}(\mathcal I_S,A)}{{\rm B}^2_S(A)}
\longrightarrow {\rm H}^2_S(A)
\longrightarrow{\rm Hom}(\mathcal K_S,A)
\longrightarrow0,
\\ \label{eq:SequenceHom(I_S)/B^2}  
&{\rm Ext}(G,A)
\longrightarrow\frac{{\rm Hom}(\mathcal I_S,A)}{{\rm B}^2_S(A)}
\longrightarrow{\rm Ext}(\mathcal C_S,A)
\longrightarrow0, 
\end{align}
where the kernel of the map ${\rm Ext}(G,A)
\to{\rm Hom}(\mathcal I_S,A)/{\rm B}^2_S(A)$
consists of the classes of symmetric two-cocycles whose restriction to $S$ is in ${\rm B}^2_S(A)$.
Moreover,~\eqref{eq:SequenceH^2_S} is split.
\end{theorem}

\begin{proof}
Consider the short exact sequence
$0\to\mathcal K_S\to\mathcal Z_S\;{\buildrel f\over\to}\;\mathcal I_S\to0$.
It splits since $\mathcal I_S$ is free, being a subgroup of $\mathcal Z_G$,
which is a free abelian group by Corollary~\ref{cor:F_G/R_GIsFree}. 
This gives the following split exact sequence:
\begin{equation}\label{eq:Sequence1}
0
\to{\rm Hom}(\mathcal I_S,A)
\to{\rm Hom}(\mathcal Z_S,A)
\to{\rm Hom}(\mathcal K_S,A)
\to0.
\end{equation}
The short exact sequence $0\to\mathcal I_S\to\mathcal Z_G\to\mathcal C_S\to0$
induces the long exact sequence in cohomology, see, e.g.,~\cite[Chapter~II.4]{macl1963A},
which, given ${\rm Ext}(\mathcal Z_G,A)=0$, takes the following form:
\begin{equation}\label{eq:Sequence2}
0
\to{\rm Hom}(\mathcal C_S,A)
\to{\rm Hom}(\mathcal Z_G,A)
\to{\rm Hom}(\mathcal I_S,A)
\to{\rm Ext}(\mathcal C_S,A)
\to0.
\end{equation}

The composition of ${\rm Hom}(\mathcal Z_G,A)\to{\rm Hom}(\mathcal I_S,A)$ from~\eqref{eq:Sequence2}
and ${\rm Hom}(\mathcal I_S,A)\to{\rm Hom}(\mathcal Z_S,A)$ from~\eqref{eq:Sequence1}  
is the map $f^*$,
which corresponds to the restriction of symmetric two-cocycles
from $\mathcal P_G$ to~$S$.
By definition, it maps ${\rm B}^2(G,A)$ onto ${\rm B}^2_S(A)$.
Recall that ${\rm Hom}(\mathcal Z_G,A)/{\rm B}^2(G,A)\simeq{\rm Ext}(G,A)$,
so sequence~\eqref{eq:Sequence2} gives us~\eqref{eq:SequenceHom(I_S)/B^2} .
Since ${\rm H}^2_S(A)={\rm Hom}(\mathcal Z_S,A)/{\rm B}^2_S(A)$,
sequence~\eqref{eq:Sequence1} gives us~\eqref{eq:SequenceH^2_S}.
\end{proof}

Note that for contractions without zeros,
Theorem~\ref{thm:ClassificationOfContractions} implies that ${\rm H}^2_S(A)={\rm Ext}(G,A)$ as expected, 
since in this case $S=\mathcal P_G$ and therefore $f$ is the identity map.

\begin{remark}\label{rm:I_S/B^2_S}
In the terminology of~\cite{weim2006a}, the subgroup ${\rm Hom}(\mathcal I_S,\mathbb F^\times)/{\rm B}^2_S(\mathbb F^\times)$ 
in sequence~\eqref{eq:SequenceH^2_S} corresponds to the classes of contractions with support $S$ 
that satisfy all surviving higher-order identities, 
i.e., take value $1$ when applied to the elements of the subgroup $\langle\mathcal R_G\rangle\cap\mathcal F_S$ of $\mathcal F_S$.
\end{remark}

The following corollary was obtained in the case $\mathbb F=\mathbb C$ in~\cite[Theorem~7.1]{weim2006a}.

\begin{corollary}\label{cor:AlgClosedH^2_S}
If $\mathbb F$ is algebraically closed then
${\rm H}^2_S(\mathbb F^\times)\simeq{\rm Hom}(\mathcal K_S,\mathbb F^\times)$.
\end{corollary}

\begin{proof}
Since $\mathbb F^\times$ is a divisible group,
both ${\rm Ext}(G,\mathbb F^\times)$ and ${\rm Ext}(\mathcal C_S,\mathbb F^\times)$ are trivial,
so sequence~\eqref{eq:SequenceHom(I_S)/B^2} implies ${\rm Hom}(\mathcal I_S,A)/{\rm B}^2_S(A)=0$.
Consequently, sequence~\eqref{eq:SequenceH^2_S} gives the result.
\end{proof}

To treat the real case,
we recall some standard results from homological algebra. 
For any abelian group $H$ (which we write multiplicatively),
we use the following notation:
\[
H_{[n]}:=\{h\in H\mid h^n=e\}
\quad\mbox{and}\quad
H^{[n]}:=\{h^n\mid h\in H\}.
\]
Then ${\rm Ext}(H,\mathbb Z/n\mathbb Z)$ is isomorphic
to the dual group, ${\rm Hom}(H_{[n]},\mathbb Q/\mathbb Z)$, of the $n$-periodic part of~$H$.
This can be seen, for example, from the injective resolution 
\[
\mathbb Z/n\mathbb Z\to\mathbb Q/\mathbb Z\to\mathbb Q/\mathbb Z\to1,
\]
where the map $\mathbb Q/\mathbb Z\to\mathbb Q/\mathbb Z$ is $x\mapsto nx$ (writing additively).
It gives that ${\rm Ext}(H,\mathbb Z/n\mathbb Z)$ is isomorphic to the quotient
${\rm Hom}(H,\mathbb Q/\mathbb Z)/{\rm Hom}(H,\mathbb Q/\mathbb Z)^{[n]}$,
which is in turn isomorphic to ${\rm Hom}(H_{[n]},\mathbb Q/\mathbb Z)$,
by applying the exact functor ${\rm Hom}(\cdot,\mathbb Q/\mathbb Z)$ to the exact sequence
\[
1\to H_{[n]}\to H\to H.
\]

If $\mathbb F$ is a real closed field (for example, $\mathbb R$), then the multiplicative group of $\mathbb F$ is the direct sum of $\{\pm1\}\simeq\mathbb Z/2\mathbb Z$ and the divisible group of positive elements $\mathbb F_{>0}$, hence 
the group ${\rm Ext}(H,\mathbb F^\times)$ is isomorphic to the dual group of $H_{[2]}$.
Then sequence~\eqref{eq:SequenceHom(I_S)/B^2} takes the form
\begin{equation}\label{eq:SequenceForSignInvar}
{\rm Hom}(G_{[2]},\{\pm1\})
\longrightarrow\frac{{\rm Hom}(\mathcal I_S,\mathbb F^\times)}{{\rm B}^2_S(\mathbb F^\times)}
\longrightarrow
{\rm Hom}((\mathcal C_{S})_{[2]},\{\pm1\})
\longrightarrow0.
\end{equation}

\begin{corollary}\label{cor:RealClosedH^2_S}
If $\mathbb F$ is real closed then
\begin{equation}\label{eq:RealContractionsClassification}
{\rm H}^2_S(\mathbb F^\times)\simeq
\frac{{\rm Hom}(\mathcal I_S,\{\pm1\})}{{\rm B}^2_S(\{\pm1\})}\oplus
{\rm Hom}(\mathcal K_S,\mathbb F^\times),
\end{equation} 
where ${{\rm Hom}(\mathcal I_S,\{\pm1\})}/{{\rm B}^2_S(\{\pm1\})}$ is an elementary 
abelian $2$-group. If $G$ is finite, then the rank of this group (i.e., dimension as a vector space 
over the field of $2$ elements) is at most $\mathop{\rm rank}G_{[2]}+N'$.
\end{corollary}

\begin{proof}
Since $G_{[2]}$ and $(\mathcal C_S)_{[2]}$ are elementary $2$-groups,
\eqref{eq:SequenceForSignInvar} implies that 
${\rm Hom}(\mathcal I_S,\mathbb F^\times)/{\rm B}^2_S(\mathbb F^\times)$ has period at most $4$.
On the other hand, using the direct sum decomposition of $\mathbb F^\times$, we get that
\[
\frac{{\rm Hom}(\mathcal I_S,\mathbb F^\times)}{{\rm B}^2_S(\mathbb F^\times)}\simeq
\frac{{\rm Hom}(\mathcal I_S,\{\pm1\})}{{\rm B}^2_S(\{\pm1\})}\oplus
\frac{{\rm Hom}(\mathcal I_S,\mathbb F_{>0})}{{\rm B}^2_S(\mathbb F_{>0})}.
\]
Since $\mathcal I_{S}$ is a free abelian group, 
the second summand above is a divisible group and hence cannot have finite period unless it is trivial.
Therefore, 
\[
{\rm Hom}(\mathcal I_S,\mathbb F^\times)/{\rm B}^2_S(\mathbb F^\times)\simeq
\frac{{\rm Hom}(\mathcal I_S,\{\pm1\})}{{\rm B}^2_S(\{\pm1\})}
\]
is an elementary 2-group.
If $N:=|G|$ is finite then, in view of sequence~\eqref{eq:SequenceForSignInvar},
the rank of this group is at most $\mathop{\rm rank}G_{[2]}+N'$,
because $\mathcal C_S$ is isomorphic to the quotient of the free abelian group $\mathcal Z_G=\mathcal F_G/\langle\mathcal R_G\rangle$ of rank $N$ by its subgroup $\mathcal I_S\simeq\mathcal F_S/(\langle\mathcal R_G\rangle\cap\mathcal F_S)$ of rank $N'$.
The split sequence~\eqref{eq:SequenceH^2_S} completes the proof.
\end{proof}

\looseness=-1
\begin{remark}\label{rem:SignInvariants}
Recall from Remark \ref{rm:I_S/B^2_S} that, in the terminology of~\cite{weim2006a},
the first summand in equation~\eqref{eq:RealContractionsClassification}
corresponds to the classes of contractions with support $S$ that satisfy all surviving higher-order identities.
Such classes are described, for $\mathbb F=\mathbb R$, in~\cite[Theorem~7.1]{weim2006a}
in terms of ``sign invariants''.
In sequence~\eqref{eq:SequenceForSignInvar},
the quotient ${\rm Hom}((\mathcal C_S)_{[2]},\{\pm1\})$
corresponds to the ``sign invariants of the first kind'' as follows.
Fix a set-theoretic section $\xi$ of the projection $\mathcal Z_G\to\mathcal C_S$. 
Given a contraction~$\varepsilon$ that satisfies all surviving higher-order identities,
i.e., $\varepsilon\in{\rm Hom}(\mathcal I_S,\mathbb R^\times)$,
we define the two-cocycle
$\alpha_\varepsilon\in {\rm Z}^2_{\rm sym}(\mathcal C_S,\{\pm1\})$ by
\[
\alpha_\varepsilon(c_1,c_2):=\sgn\big(\varepsilon(\xi(c_1)\xi(c_2)\xi(c_1c_2)^{-1})\big),
\quad
c_1,c_2\in\mathcal C_S,
\]
where $\sgn$ denotes the usual sign function,
which is the projection $\mathbb R^\times\twoheadrightarrow\{\pm1\}$.
The class $[\alpha_\varepsilon]\in{\rm Ext}(\mathcal C_S,\{\pm1\})$
is the image of $[\varepsilon]$ under the homomorphism 
$
{\rm Hom}(\mathcal I_S,\mathbb R^\times)/{\rm B}^2_S(\mathbb R^\times)
\to{\rm Ext}(\mathcal C_S,\{\pm1\})
$
from sequence~\eqref{eq:SequenceHom(I_S)/B^2}.
The corresponding element $\delta_\varepsilon\in{\rm Hom}((\mathcal C_S)_{[2]},\{\pm1\})$
is given by
\[
\delta_\varepsilon(c):=\alpha_\varepsilon(c,c)\alpha_\varepsilon(e,e)=\sgn\big(\varepsilon(\xi(c)^2)\big),
\quad
c\in\mathcal (\mathcal C_S)_{[2]}.
\]
The values of $\delta_\varepsilon$ are the ``sign invariants of the first kind'' assigned to $\varepsilon$.
\end{remark}

\section{Variety of graded contractions}\label{sec:VarietyOfContractions}

Now we approach $G$-graded contractions from the perspective of commutative algebra and affine algebraic geometry.
More specifically, we consider condition $(ii')$ of Proposition~\ref{prop:GenericContraction}
as a relation in the ring of polynomials $\mathbb F[x_i\mid i\in\mathcal P_G]$,
whose variables are labeled, in view of condition $(i')$, by the set $\mathcal P_G$ of unordered pairs $\{g,h\}\subset G$.
To avoid technical issues, the ground field $\mathbb F$ in this section is assumed to be algebraically closed.

Recall that Zariski topology on the affine space $\mathbb F^n$, $n\in\mathbb N$,
is defined by taking as closed sets all {\it algebraic sets},
i.e., a set $V\subset\mathbb F^n$ is Zariski closed if it is the set 
of common zeros of a collection of polynomials from $R:=\mathbb F[x_1,\dots,x_n]$.
These polynomials generate an ideal $J$ of $R$, and $V$ is the vanishing set of $J$,
$V(J):=\{x\in\mathbb F^n\mid f(x)=0\ \forall f\in I\}$.
Conversely, for any set $S\subset\mathbb F^n$, we can consider its vanishing ideal,
$I(S):=\{f\in R\mid f(x)=0\ \forall x\in S\}$.
The set $V(I(S))$ is the Zariski closure of $S$,
and the cornerstone result in affine algebraic geometry, Hilbert's Nullstellensatz,
says that, for any ideal $J$ of $R$, $I(V(J))=\sqrt J$,
where the ideal $\sqrt J:=\{f\in R\mid \exists n\in\mathbb N\colon f^n\in J\}$
is called the {\it radical} of $J$.
Thus we have a one-to-one correspondence between Zariski closed sets and {\it radical ideals} $I$ of the polynomial ring $R$,
i.e., the ideals satisfying $I=\sqrt I$.

By definition, an affine algebraic variety $X$ is isomorphic to a Zariski closed set in the affine space $\mathbb F^n$
for some $n\in\mathbb N$, i.e., $X\simeq V(J)$ for some ideal $J\subset R$.
The geometric properties of $X$ can be studied through the algebra of polynomial functions restricted to $X$,
known as the {\it algebra of regular functions} and denoted $\mathbb F[X]:=R/\sqrt J$.
For more details, see standard references like~\cite{cox2015A,hart1977A}.

\medskip\par\noindent
{\it Cellular decomposition of the variety of graded contractions.}
Let $G$ be a finite abelian group and consider the affine space $\mathbb A_G$,
whose coordinates are indexed by the set $\mathcal P_G$,
so the coordinate algebra is $\mathbb F[\mathbb A_G]=\mathbb F[x_i\mid i\in\mathcal P_G]$.
Then the set $X$ of all $G$-graded contractions can be regarded as a subset of $\mathbb A_G$.
Moreover, $X$ is an affine algebraic variety determined by the equations
\begin{equation}\label{eq:Relations}
x_{\{g,h\}}x_{\{gh,k\}}-x_{\{h,k\}}x_{\{g,hk\}}=0\quad
\mbox{for all}\quad
g,h,k\in G.
\end{equation}
Let $I\subset\mathbb F[\mathbb A_G]$ be the ideal generated by the left-hand sides  of the above equations.
This ideal is {\it binomial},
meaning it is generated by polynomials with two terms.
Binomial ideals have a number of remarkable properties,
for instance, the radical and the associated primes of a binomial ideal are again binomial.
Eisenbud and Sturmfels~\cite{eise1996a} provide a detailed study of these and other properties
and characterize algebraic sets defined by binomial ideals in terms of their cellular decomposition,
see~\cite[Theorem~4.2]{eise1996a}.
We will subsequently study the cellular decomposition of the variety~$X$ of graded contractions.
It should be noted that the ideal $I$ is usually not radical,
which can be verified using {\sc Singular} computer algebra system~\cite{deck2024A}.
For instance, $I$ is radical when $G$ is of order $2$, $3$, $4$ or $5$, but not $6$. 

For any subset $ S\subset\mathcal P_G$,
we consider the affine subspace $\mathbb A_S:=\{x\in\mathbb A_G\colon x_i\ne0~\forall i\in S\}\subset\mathbb A_G$
and the torus ${\rm T}_ S:=\{x\in\mathbb A_S\colon x_i\ne0~\forall i\in S\}$,
which is a principal open set in $\mathbb A_S$.
Since $x\in{\rm T}_S$ if and only if $\mathop{\rm supp}x= S$,
where $\mathop{\rm supp}x=\{i\in\mathcal P_G\colon x_i\ne0\}$,
we have a partition $\mathbb A_S=\bigcup_{S'\subset S}{\rm T}_{S'}$.
The tori ${\rm T}_S$ are called {\it coordinate cells} in~\cite{eise1996a}.
The intersection $X\cap{\rm T}_S$ consists precisely of the graded contractions
whose support is $S$;
it is nonempty if and only if $S\in\mathcal S(G)$.
It follows that we have the following partition of $X$:
\[
X=\bigcup_{S\in\mathcal S(G)}X_S,\quad
X_S:=X\cap{\rm T}_S.
\]

In particular,  $X_{\mathcal P_G}=X\cap {\rm T}_{\mathcal P_G}={\rm Z}^2_{\rm sym}(G,\mathbb F^\times)$ is a nonempty principal open set in $X$.
Consider the torus ${\rm T}_G:={\rm Fun}(G,\mathbb F^\times)$
and its action on the space $\mathbb A_G$ defined by
\begin{equation*}
\alpha\cdot(x_{\{g,h\}})_{\{g,h\}\in\mathcal P_G}=\big(d\alpha(g,h)\,x_{\{g,h\}}\big)_{\{g,h\}\in\mathcal P_G},
\end{equation*}
where recall that $d\alpha(g,h)=\alpha(g)\alpha(h)/\alpha(gh)$.
In other words, ${\rm T}_G$ acts through the homomorphism $d\colon{\rm T}_G\to {\rm T}_{\mathcal P_G}$,
$\alpha\mapsto d\alpha$. 
This action leaves $X$ and each cell $X_S$ invariant;
moreover, the group ${\rm H}^2_S(\mathbb F^\times)$, which was introduced in the previous section,
is in bijection with the set of orbits of $X_S$ under this action.
Note that the ${\rm T}_G$-orbit of the point $\bar 1:=(1,\dots,1)$
is ${\rm B}^2(G,\mathbb F^\times)={\rm Z}^2_{\rm sym}(G,\mathbb F^\times)=X_{\mathcal P(G)}$.

\begin{proposition}\label{prop:X(1)Irred}
The closure of the ${\rm T}_G$-orbit of $\bar 1$, which we denote by $X(\bar 1)$,
is the unique irreducible component of $X$ containing $\bar 1$.
Moreover, the dimension of $X(\bar 1)$ is $N=|G|$. 
\end{proposition}

\begin{proof}
Since~${\rm T}_G$ is irreducible, so is~$X(\bar 1)$,
hence it is contained in an irreducible component of~$X$.
But it contains an open neighborhood of~$\bar 1$ in~$X$,
so it must contain any irreducible component of~$X$ passing through~$\bar 1$.

Since $\dim {\rm T}_G=N$ and the stabilizer of $\bar1$, ${\rm Hom}(G,\mathbb F^\times)$, is finite, 
the second assertion follows from the fact that the ${\rm T}_G$-orbit of $\bar 1$
is open in its closure, $X(\bar 1)$, see Section~\ref{sec:ContAndDegenOfLieAlgs}.
\end{proof}

\begin{example}\label{ex:Z2}
Consider the grading group $G=\mathbb Z_2$.
The affine space $\mathbb A_{\mathbb Z_2}$ is coordinatized by the tuple $(x_{\{0,0\}},x_{\{1,1\}},x_{\{0,1\}})$
and its coordinate algebra is $\mathbb C[x_{\{0,0\}},x_{\{1,1\}},x_{\{0,1\}}]$.
The variety of contractions is the zero locus of the polynomials
\[
x_{\{0,0\}}x_{\{0,1\}}-x_{\{0,1\}}^2,\quad
x_{\{0,0\}}x_{\{1,1\}}-x_{\{1,1\}}x_{\{0,1\}}.
\]
The ideal~$I$ generated by this polynomial is actually radical.
The primary decomposition of $I$ is given by
\begin{gather*}
I=\langle x_{\{0,0\}}-x_{\{0,1\}}\rangle\cap
\langle x_{\{1,1\}},x_{\{1,0\}}\rangle.
\end{gather*}
The variety $X=V(I)$ is the union $V(x_{\{0,0\}}-x_{\{0,1\}})\cup V(x_{\{1,1\}},x_{\{1,0\}})$
of the plane $x_{\{0,0\}}-x_{\{0,1\}}=0$ and the line $x_{\{1,1\}}=x_{\{1,0\}}=0$.
The supports of $\mathbb Z_2$-graded contractions are the sets
\begin{gather*}
S_1:=\varnothing,\quad
S_2:=\{x_{\{0,0\}}\},\quad
S_3:=\{x_{\{1,1\}}\},
\\
S_4:=\{x_{\{0,0\}},x_{\{0,1\}}\},\quad
S_5:=\{x_{\{0,0\}},x_{\{1,1\}},x_{\{0,1\}}\}.
\end{gather*}
The cellular decomposition of $X$ is given by
\begin{gather*}
X\cap{\rm T}_{S_1}=(0,0,0),\quad
X\cap{\rm T}_{S_2}=\big(V(x_{\{0,0\}}-x_{\{0,1\}})\cup V(x_{\{1,1\}},x_{\{1,0\}})\big)\cap{\rm T}_{S_2}
\\
X\cap{\rm T}_{S_i}=V(x_{\{0,0\}}-x_{\{0,1\}})\cap{\rm T}_{S_i},\quad
i=3,4,5.
\end{gather*}
The ${\rm T}_{\mathbb Z_2}$-orbit of the point $(1,1,1)$, which corresponds to the identity function $\bar1$,
is the cell $X\cap{\rm T}_{S_5}$ of $X$.
Its closure is the plane $V(x_{\{0,0\}}-x_{\{0,1\}})$.
\end{example}

\begin{example}\label{ex:Z3}
The computations for the grading group $G=\mathbb Z_3$ are more complicated than those
presented in Example~\ref{ex:Z2}.
The coordinate algebra of $\mathbb A_{\mathbb Z_3}$ is the polynomial algebra in six variables,
$\mathbb C[x_{\{0,0\}},x_{\{1,1\}},x_{\{2,2\}},x_{\{0,1\}},x_{\{0,2\}},x_{\{1,2\}}]$.
The variety $X$ of contractions is the zero locus of the polynomials
\begin{gather*}
x_{\{0,0\}}x_{\{0,1\}}-x_{\{0,1\}}^2,\quad
x_{\{0,0\}}x_{\{0,2\}}-x_{\{0,2\}}^2,\quad
x_{\{0,0\}}x_{\{1,2\}}-x_{\{0,1\}}x_{\{1,2\}},\\
x_{\{0,0\}}x_{\{1,2\}}-x_{\{0,2\}}x_{\{1,2\}},\quad
x_{\{1,1\}}x_{\{0,1\}}-x_{\{1,1\}}x_{\{0,2\}},\quad
x_{\{2,2\}}x_{\{0,1\}}-x_{\{2,2\}}x_{\{0,2\}},\\
x_{\{0,1\}}x_{\{1,2\}}-x_{\{0,2\}}x_{\{1,2\}},\quad
x_{\{1,1\}}x_{\{2,2\}}-x_{\{0,1\}}x_{\{1,2\}},\quad
x_{\{1,1\}}x_{\{2,2\}}-x_{\{0,2\}}x_{\{1,2\}}.
\end{gather*}
The primary decomposition of $I$ is given by
\begin{gather*}
I=\langle x_{\{0,1\}}-x_{\{0,2\}},\,x_{\{1,1\}}x_{\{2,2\}}-x_{\{0,2\}}x_{\{1,2\}},\,x_{\{0,0\}}-x_{\{0,2\}}\rangle
\\ \qquad
\cap\langle x_{\{1,2\}},x_{\{0,2\}},x_{\{0,1\}},x_{\{2,2\}}\rangle
\cap\langle x_{\{1,2\}},x_{\{0,2\}},x_{\{0,1\}},x_{\{1,1\}}\rangle
\\ \qquad
\cap\langle x_{\{1,2\}},x_{\{0,1\}},x_{\{2,2\}},x_{\{1,1\}},x_{\{0,0\}}-x_{\{0,2\}}\rangle
\\ \qquad
\cap\langle x_{\{1,2\}},x_{\{0,2\}},x_{\{2,2\}},x_{\{1,1\}},x_{\{0,0\}}-x_{\{0,1\}}\rangle.
\end{gather*}
The irreducible components of the variety $X$ are
\begin{itemize}
\item two lines:
$V(x_{\{1,2\}},\,x_{\{0,2\}},\,x_{\{2,2\}},\,x_{\{1,1\}},\,x_{\{0,0\}}-x_{\{0,1\}})$ and
$V(x_{\{1,2\}},\,x_{\{0,1\}},\,x_{\{2,2\}},\,x_{\{1,1\}},\newline  x_{\{0,0\}}-x_{\{0,2\}})$,
\item two planes:
$V(x_{\{1,2\}},x_{\{0,2\}},x_{\{0,1\}},x_{\{1,1\}})$ and
$V(x_{\{1,2\}},x_{\{0,2\}},x_{\{0,1\}},x_{\{2,2\}})$,
\item dimension-three variety:
$V(x_{\{0,1\}}-x_{\{0,2\}},\,x_{\{1,1\}}x_{\{2,2\}}-x_{\{0,2\}}x_{\{1,2\}},\,x_{\{0,0\}}-x_{\{0,2\}})$.
\end{itemize}
\end{example}

\medskip\par\noindent
{\it Finding supports.}
We now describe an algorithm to find the elements of $\mathcal S(G)$ and count their number.
Let $J$ be the ideal in the polynomial ring $\mathbb C[x_{\{g,h\}}\mid g,h\in G]=\mathbb C[\mathbb A_G]$
generated by the binomials
\[
x_{\{g,h\}}x_{\{gh,k\}}-x_{\{h,k\}}x_{\{g,hk\}},\quad
x_{\{g,h\}}(x_{\{g,h\}}-1)\quad
\mbox{for all}\quad
g,h,k\in G.
\]
The affine variety~$V(J)$ is a finite set of points in the affine space~$\mathbb A_G$.
These points satisfy the determining equations for $G$-graded contractions and their coordinates are zero or one,
so they correspond to the elements of $\mathcal S(G)$.

Now consider the algebra $\mathbb C[\mathbb A_G]/\sqrt J$ of regular functions of $V(J)$.
Since~$V(J)$ is merely a finite set of points, the algebra $\mathbb C[\mathbb A_G]/\sqrt J$ is finite-dimensional
and its dimension is precisely the number of points in~$V(J)$.
\noprint{Furthermore, the image of the set of monomials $\{\prod_{\{g,h\}\in S} x_{\{g,h\}}\mid S\subset\mathcal P_G\}$
under the projection $\mathbb C[\mathbb A_G]\twoheadrightarrow\mathbb C[\mathbb A_G]/\sqrt J$
is the canonical basis of the algebra $\mathbb C[\mathbb A_G]/\sqrt J$.
The elements of this basis correspond precisely to the indicator functions of the supports.}

All the above computations can be efficiently carried out in {\sc Singular}.
Even for the order-8 grading groups, i.e., the number of variables of $\mathbb C[\mathbb A_G]$ is $36$,
these computations are performed within seconds.
In Table~\ref{tab:Supports} we present the numerical results for the number of supports depending on the grading group $G$.
We also observe an interesting phenomenon that only a small part of subsets of $\mathcal P_G$ are actually supports.
\begin{table}[!ht]\footnotesize
\caption{Number of supports}
\label{tab:Supports}
\renewcommand{\arraystretch}{1.4}
\begin{center}
\begin{tabular}{|c|c|c|c|c|c|c|c|c|c|c|}
\hline
$G$                          & $\mathbb Z_2$  & $\mathbb Z_3$ & $\mathbb Z_2^2$ & $\mathbb Z_4$ & $\mathbb Z_5$
&$\mathbb Z_6$ &$\mathbb Z_7$ & $\mathbb Z_2^3$  & $\mathbb Z_2\times\mathbb Z_4$ & $\mathbb Z_8$\\
\hline
$|\mathcal S(G)|$                     & 5              & 15            & $41$                           & $47$          & $185$ 
& $652$ & $2997$ &  $16147$            & $14179$          & $13883$ \\
$2^{|\mathcal P(G)|}$        & 8              & $2^6$         & $2^{10}$                       & $2^{10}$      & $2^{15}$
& $2^{21}$ & $2^{28}$ &  $2^{36}$           & $2^{36}$         & $2^{36}$\\
$\frac{|\mathcal S(G)|}{2^{|\mathcal P(G)|}}$ & 0.625          & $0.2344$      & $0.040$                         & $0.046$       & $1.43{\rm E}\mbox{-3}$
& $3.11{\rm E}\mbox{-4}$ & $1.12{\rm E}\mbox{-5}$ & $2.35{\rm E}\mbox{-7}$  & $2.06{\rm E}\mbox{-7}$  &  $2.02{\rm E}\mbox{-7}$   \\ 
\hline
\end{tabular}
\end{center}
\end{table}

\medskip\par\noindent
{\it Orbit closures and higher-order identities}.
The elements of $X(\bar 1)$ are precisely the {\it continuous contractions} as defined in~\cite[Section~4]{mood1991a}.
In the case $\mathbb F=\mathbb C$,
$\mathfrak L^\varepsilon$ is a graded degeneration of a Lie algebra $\mathfrak L$ for any $\varepsilon\in X(\bar 1)$,
so, in particular, it is a contraction of $\mathfrak L$ in the sense of Definition~\ref{def:ContinuousContraction}.
Recall from Section~\ref{sec:GroupCohomology}
that every relation $r=1$ for $r\in\langle\mathcal R_G\rangle$
can be uniquely rewritten in the form $r_1=r_2$,
where $r_1$ and $r_2$ are monomials of the same degree in two disjoint sets of variables.
Following~\cite{weim1995a,weim2006a}, we call these relations {\it higher-order identities};
they include the defining relations~\eqref{eq:Relations}.
Since the elements of ${\rm Z}^2_{\rm sym}(G,\mathbb F^\times)$ satisfy all higher-order identities,
so do the elements of the closure, $X(\bar 1)$.
It follows from Theorem~\ref{thm:ClassificationOfContractions} and~\cite[Theorem~VI.2]{weim1995a} that any element of $X(\bar 1)$
is equivalent via normalization to a generalized In\"on\"u--Wigner contraction.

Now let $S$ be a subset of $\mathcal P_G$ and $\bar 1_S$ be its indicator function
considered as the $|\mathcal P_G|$-tuple:
\begin{equation*}
(\bar 1_S)_i:=
\begin{cases}
1 &\mbox{if }i\in S,\\
0 &\mbox{otherwise}.
\end{cases}
\end{equation*}
Note that $S\in\mathcal S(G)$ if and only if $\bar1_S$ belongs to~$X$.
In this case denote by $X(\bar1_S)\subset X$ the closure of the ${\rm T}_G$-orbit of $\bar1_S$.
In particular $S=\mathcal P_G$ gives~$X(\bar1)$.
More generally, for any $\varepsilon\in X_S$, denote by $X(\varepsilon)$ the closure of ${\rm T}_G$-orbit of $\varepsilon$.
Note that $X(\varepsilon)\subset \bigcup_{S'\subset S}X_{S'}$.
The following result describes the vanishing ideal of $X(\varepsilon)$
in terms of the {\it surviving higher-order identities},
i.e., those of the form $r_1=r_2$ where both sides involve variables indexed only by $i\in S$.

\begin{theorem}\label{thm:VanishingIdeal}
Let $G$ be an abelian group and $S$ be a finite subset of $\mathcal P_G$.
Let $\mathbb F$ be an algebraically closed field and consider the action of
${\rm Fun}(G,\mathbb F^\times)$ on $\mathbb A_S$ given by
\[
\alpha\cdot(x_{\{g,h\}})_{\{g,h\}\in S}=\big(d\alpha(g,h)\,x_{\{g,h\}}\big)_{\{g,h\}\in S}.
\]
Then the vanishing ideal of the closure of the orbit of $\bar 1_S$ in $\mathbb A_S$
is spanned by the binomials $r_1-r_2$ arising from all surviving higher-order identities,
i.e., satisfying $r_1r_2^{-1}\in\langle\mathcal R_G\rangle\cap\mathcal F_S$.
More generally, the vanishing ideal of the closure of the orbit of any $\varepsilon\in{\rm T}_S={\rm Fun}(S,\mathbb F^\times)$
is spanned by the binomials $\varepsilon(r_2)r_1-\varepsilon(r_1)r_2$,
where we make the identification ${\rm Fun}(S,\mathbb F^\times)={\rm Hom}(\mathcal F_S,\mathbb F^\times)$.
\end{theorem}

\begin{proof}
The orbit of $\bar 1_S$ is the range of the function $f_S\colon{\rm Fun}(G,\mathbb F^\times)\to\mathbb A_S$, $\alpha\mapsto \alpha\cdot\bar 1_S$.
Since $\alpha\cdot\bar 1_S=(d\alpha)\bar 1_S$, 
the range of $f_S$ is precisely ${\rm B}^2_S(\mathbb F^\times)$ defined in Section~\ref{sec:GroupCohomology}.
For any monomial~$r$ in variables indexed by $i\in S$,
denote by~$\lambda_r$ the restriction of the corresponding polynomial function $\mathbb A_S\to\mathbb F$
to the torus ${\rm Fun}(S,\mathbb F^\times)$.
First we claim that $\lambda_{r_1}|_{{\rm B}^2_S(\mathbb F^\times)}=\lambda_{r_2}|_{{\rm B}^2_S(\mathbb F^\times)}$ if and only if $r_1r_2^{-1}\in\mathcal\langle\mathcal R_G\rangle\cap\mathcal F_S$.

In the course of the proof we use the following notation:
if $B$ is a subgroup of an abelian group $A$ (written multiplicatively),
then
\[
B^\perp:=\{f\in{\rm Hom}(A,\mathbb F^\times)\mid f(B)=1\}
\subset{\rm Hom}(A,\mathbb F^\times).
\]
Similarly, if $H$ is a subgroup of ${\rm Hom}(A,\mathbb F^\times)$,
then
\[
H^\perp:=\{a\in A\mid \chi(a)=1\mbox{ for all }\chi\in H\}\subset A.
\]

It is clear that $\lambda_{r_1}|_{{\rm B}^2_S(\mathbb F^\times)}=\lambda_{r_2}|_{{\rm B}^2_S(\mathbb F^\times)}$ if and only if the function $\lambda_{r_1}/\lambda_{r_2}$
maps ${\rm B}^2_S(\mathbb F^\times)$ to~$1$,
and the latter condition is equivalent to $r_1r_2^{-1}\in {\rm B}^2_S(\mathbb F^\times)^\perp\subset\mathcal F_S$.

Theorem~\ref{thm:ClassificationOfContractions} and Corollary~\ref{cor:AlgClosedH^2_S}
give us that the kernel of the restriction map
${\rm Hom}(\mathcal Z_S,\mathbb F^\times)\to{\rm Hom}(\mathcal K_S,\mathbb F^\times)$
is ${\rm B}^2_S(\mathbb F^\times)$.
Since $\mathcal Z_S=\mathcal F_S/\langle\mathcal R_S\rangle$
and $\mathcal K_S=(\langle\mathcal R_G\rangle\cap\mathcal F_S)/\langle\mathcal R_S\rangle$,
the kernel of the restriction map
${\rm Hom}(\mathcal F_S,\mathbb F^\times)\to{\rm Hom}(\langle\mathcal R_G\rangle\cap\mathcal F_S,\mathbb F^\times)$
is also ${\rm B}^2_S(\mathbb F^\times)$.
By definition, this kernel is $(\langle\mathcal R_G\rangle\cap\mathcal F_S)^\perp$,
so we get ${\rm B}^2_S(\mathbb F^\times)=(\langle\mathcal R_G\rangle\cap\mathcal F_S)^\perp$ and hence ${\rm B}^2_S(\mathbb F^\times)^\perp=((\langle\mathcal R_G\rangle\cap\mathcal F_S)^\perp)^\perp$.
Finally, $\mathcal F_S/(\langle\mathcal R_G\rangle\cap\mathcal F_S)\simeq\mathcal I_S$ is a free abelian group,
so its points are separated by homomorphisms to~$\mathbb F^\times$.
Therefore, $((\langle\mathcal R_G\rangle\cap\mathcal F_S)^\perp)^\perp=\langle\mathcal R_G\rangle\cap\mathcal F_S$,
which completes the proof of the claim.

Let $I$ be the vanishing ideal of the closure of ${\rm B}^2_S(\mathbb F^\times)$ in $\mathbb A_S$.
Since ${\rm B}^2_S(\mathbb F^\times)$ is dense in its closure,
we have
\[
I=\{\varphi\in\mathbb F[\mathbb A_S]\mid \varphi|_{{\rm B}^2_S(\mathbb F^\times)}=0\}.
\]
A binomial $r_1-r_2$ arises from a surviving higher-order identity
if and only if $r_1r_2^{-1}\in\langle\mathcal R_G\rangle\cap\mathcal F_S$.
It follows from the above claim that
all these binomials $r_1-r_2$ belong to $I$.
Modulo a linear combination of such binomials,
any $\varphi\in\mathbb F[\mathbb A_S]$ 
is congruent to a linear combination of monomials $r_j$ such that 
$r_ir_j^{-1}\notin\langle\mathcal R_G\rangle\cap\mathcal F_S$ for $i\ne j$.
But the restrictions $\lambda_{r_j}|_{{\rm B}^2_S(\mathbb F^\times)}$
are pairwise distinct multiplicative characters of ${\rm B}^2_S(\mathbb F^\times)$,
hence they are linearly independent over $\mathbb F$.
Therefore the ideal $I$ is spanned by the binomials $r_1-r_2$
arising from surviving higher-order identities.
This finishes the proof of the first claim of the theorem

The second part of theorem's statement follows, since the ${\rm T}_S$-orbit of $\varepsilon$
is ${\rm B}^2_S(\mathbb F^\times)\varepsilon$.
\end{proof}

\begin{remark}
When $G$ is a finite group, the closure of the orbit of $\bar1_S$ is, by definition, the toric variety
arising from the integral $|G|\times|S|$ matrix $A$,
whose columns are indexed by $\{g,h\}\in S$ and are given by components of the vector $g+h-gh$
in the standard basis of the free $\mathbb Z$-module $\mathbb ZG$,
see~\cite[\S1.1]{cox2011A} for details about toric varieties.
The nullspace $\{x\in\mathbb Z^{|S|}\mid Ax=0\}$ of the matrix $A$ can be identified with the abelian group $\mathcal F_S\cap\ker\p_2$,
where $\p_\bullet$ is the differential of the complex $C_\bullet$ defined in Section~\ref{sec:GroupCohomology}.
Since $\mathop{\rm im}\p_3=\langle\mathcal R_G\rangle$ and $\mathcal F_S\cap\ker\p_2=\mathcal F_S\cap\mathop{\rm im}\p_3$
by exactness of the complex~$C_\bullet$,
Theorem~\ref{thm:VanishingIdeal} follows from~\cite[Proposition~1.1.9]{cox2011A}.
\end{remark}

\begin{corollary}\label{cor:VanishingIdeals}
Let $G$ be a finite abelian group.

(1) The vanishing ideal of $X(\bar 1)$ is spanned by the binomials $r_1-r_2$ arising from all higher-order identities.

(2) For $S\in\mathcal S(G)$, we have $X(\bar1_S)\subset X(\bar 1)$ if and only if, for any higher-order identity $r_1=r_2$,
if $r_1$ involves only variables in $S$ then so does $r_2$
(no ``weak violations'' in the terminology of~\cite{weim2006a}).

(3) For any $\varepsilon\in X_S$ and $\varepsilon'\in X$,
we have $X(\varepsilon')\subset X(\varepsilon)$ if and only if
$S(\varepsilon')\subset S$ and $\varepsilon'\varepsilon^{-1}$
satisfies all surviving higher-order identities $r_1=r_2$ for $S$,
i.e., arising from $r_1r_2^{-1}\in\langle R_G\rangle\cap\mathcal F_S$.
\end{corollary}

Theorem~\ref{thm:VanishingIdeal} is relevant when applying a $G$-graded contraction $\varepsilon$
to a finite-dimensional Lie algebra $\mathfrak L$
with a $G$-grading $\Gamma\colon\mathfrak L=\bigoplus_{g\in G}\mathfrak L_g$.
Let $S:={\rm Supp}_2(\Gamma)$,
i.e., the set 
\[
\{\{g,h\}\in\mathcal P_G\mid[\mathfrak L_g,\mathfrak L_h]\ne0\}\subset\mathcal P_{{\rm Supp}(\Gamma)}.
\]
Then the Lie bracket $\mu$ of $\mathfrak L$ is the direct sum of
$\mu_{g,h}\colon\mathfrak L_g\times\mathfrak L_h\to\mathfrak L_{gh}$,
and the bracket of the contracted algebra~$\mathfrak L^\varepsilon$ is obtained by multiplying
each $\mu_{g,h}$ by $\varepsilon(g,h)$.
Moreover, the action of $D(\Gamma)$ on $\mu$ from Section~\ref{sec:GradingsAndContractions}
corresponds to the action of ${\rm Fun}(G,\mathbb F^\times)$ on $\bar 1_S$.
Thus the bracket of $\mathcal L^\varepsilon$ is in the closure of the $D(\Gamma)$-orbit of $\mu$
if and only if $\varepsilon$ is in the closure of the ${\rm Fun}(G,\mathbb F^\times)$-orbit of $\bar 1_S$.
Theorem~\ref{thm:VanishingIdeal} implies the following:

\begin{corollary}\label{cor:GradedDegenerationsAndContractions}
The contracted algebra $\mathcal L^\varepsilon$ is a graded degeneration of $\mathcal L$
if and only if $\varepsilon$ satisfies all surviving higher-order identities for $S={\rm Supp}_2(\Gamma)$.
In particular, if ${\rm Supp}_2(\Gamma)=\mathcal P_G$, then the graded degenerations
are precisely the graded contractions in $X(\bar1)$,
i.e., they satisfy all higher-order identities.
\end{corollary}

\section{Graded contractions as lax monoidal functors}\label{sec:EquivOfContractions}

Now we will consider the way to formalize the ``generic'' point of view on contractions
using the functors introduced by Moody and Patera in \cite{mood1991a},
which we will define in a more general setting of lax monoidal functors
on the category ${\rm Vec}^G$ of $G$-graded vector spaces.
As noted in Section~\ref{sec:GradingsAndContractions},
there are different types of equivalences for graded contractions, two of which correspond to isomorphism in ${\rm Vec}^G$,
namely, isomorphism (Definition~\ref{def:StrongEquiv})
and equivalence via normalization (Definition~\ref{def:EquivViaNormalization}).
One of them implies the other, and Weimar--Woods conjectured that they, in fact, coincide:

\begin{conjecture}[\cite{weim2006a}]\label{conj:WeimWoods}
Two graded contractions $\varepsilon$ and $\varepsilon'$ are isomorphic, $\varepsilon\simeq\varepsilon'$,
if and only if they are equivalent via normalization, $\varepsilon\sim_{\rm n}\varepsilon'$.
\end{conjecture}

Draper et al. recently showed that this conjecture holds for the grading group $\mathbb Z_2^3$, by considering a special $\mathbb Z_2^3$-grading on the simple Lie algebra of type $G_2$, see discussion in~\cite[Section~1]{drap2024a}.
Note that, if we assume that the isomorphism in Definition~\ref{def:StrongEquiv} holds for a specific $G$-graded Lie algebra $\mathfrak L$, rather than {\it for all} Lie algebras,
then this statement is false, even if $\mathfrak L$ is semisimple and the support of the grading generates $G$, as discussed in~\cite[Example~2.12]{drap2024a}.
Since there exists a $G$-graded Lie algebra with the second support $\mathcal P_G$
(see Proposition~\ref{prop:SecondSupport}),
it is clear that $\varepsilon\simeq\varepsilon'$ implies $S(\varepsilon)=S(\varepsilon')$.

\begin{proposition}\label{thm:WittAlgebra}
For the group $\mathbb Z$, if $\varepsilon\simeq\varepsilon'$,
then there exists $\lambda\colon\mathbb Z\to\mathbb F^\times$ such that
\[
\varepsilon'(m,n)=\varepsilon(m,n)\,d\lambda(m,n)\quad
\mbox{for all }
m\ne n.
\]
\end{proposition}
\begin{proof}
Consider the Witt algebra $\mathcal W$ with the canonical basis given by $L_n$, $n\in\mathbb Z$,
and the Lie bracket given by
\[
[L_m,L_n]=(m-n)L_{m+n},\quad
\forall m,n\in\mathbb Z.
\]
The algebra $\mathcal W$ has a natural $\mathbb Z$-grading,
\[
\mathcal W=\bigoplus_{n\in\mathbb Z}\mathcal W_n,\quad
\mathcal W_k=\langle L_k\rangle\quad
\forall k\in\mathbb Z.
\]
Let $\varepsilon$ and $\varepsilon'$ be isomorphic $\mathbb Z$-graded contractions of $\mathcal W$.
This means that there exists an isomorphism of graded Lie algebras,
$\varphi\colon\mathcal W^\varepsilon\to\mathcal W^{\varepsilon'}$.
Since $\dim\mathcal W_n=1$,
for any $n\in\mathbb Z$ there exists $\lambda_n\in\mathbb F^\times$ such that $\varphi(L_n)=\lambda_nL_n$.
The equation
$\varphi([L_m,L_n]_\varepsilon)=[\varphi(L_m),\varphi(L_n)]_{\varepsilon'}$
then gives the result.
\end{proof}

It is natural to consider Weimar--Woods conjecture within the categorical framework,
where the notion of isomorphism is replaced with that of natural isomorphism.
In fact, Moody and Patera initiated such consideration in \cite[Section~4]{mood1991a} 
even before Conjecture~\ref{conj:WeimWoods} appeared in the literature.
They associated to graded contractions
some specific endofunctors of the category of $G$-graded Lie algebras and further observed
that the functors corresponding to contractions that are equivalent via normalization are naturally isomorphic.

We extend this analysis by providing a more general context
using monoidal categories and lax monoidal functors,
which applies not only to Lie algebras but also to Lie color algebras, in particular, to Lie superalgebras.

\medskip\par\noindent
{\it Lax monoidal functors.}
Here and in what follows we adopt the notation from~\cite{etin2015A}:
a monoidal category is given by a quintuple
$(\mathcal C,\otimes,\mathsf 1,\alpha,\iota)$,
where~$\mathcal C$ is a category, $\otimes$ denotes the tensor product bifunctor,
the natural isomorphism~$\alpha$ determines the associativity constraint,
$\mathsf 1\in\mathcal C$ is the tensor unit,
and 
$\iota\colon\mathsf 1\otimes\mathsf 1\to\mathsf 1$ is an isomorphism,
which satisfy certain axioms.
When the context is clear, we may refer to this quintuple simply as $\mathcal C$.
Furthermore, the isomorphism~$\iota$ gives rise to the unit constraints~$r$ and~$l$ \cite[Section~2.2]{etin2015A},
which may be used for an alternative definition of a monoidal category \cite[Definition~2.2.8]{etin2015A}.

A monoidal functor from $\mathcal C$ to $\mathcal C'$ is a pair $(F,J)$,
where $F\colon\mathcal C\to\mathcal C'$ is a functor
and $J_{X,Y}\colon F(X)\otimes'F(Y)\to F(X\otimes Y)$ is a natural isomorphism,
such that $F(\mathsf 1)$ is isomorphic to $\mathsf 1'$
and the diagram~\eqref{cd:Lax} below commutes.
The notion of a lax monoidal functor between monoidal categories
generalizes that of a monoidal functor 
by relaxing the condition of isomorphism with just morphism as follows:

\begin{definition}\label{def:LaxMonoidal}
Let $(\mathcal C,\otimes,\mathsf 1,\alpha,\iota)$ and 
$(\mathcal C',\otimes',\mathsf 1',\alpha',\iota')$ be monoidal categories.
A {\it lax monoidal functor} between them is a triple $(F,J,\epsilon)$,
where $F$ is a functor $F\colon\mathcal C\to\mathcal C'$ together with coherence morphisms
\begin{enumerate}\itemsep=0ex
\item 
a morphism
$\epsilon\colon\mathsf 1'\to F(\mathsf 1)$,

\item
a natural transformation 
$J_{X,Y}\colon F(X)\otimes'F(Y)\to F(X\otimes Y)$
\end{enumerate}
satisfying the following conditions
\begin{enumerate}\itemsep=0ex
\item[a.]
({\it associativity})
For all objects $X,Y,Z\in\mathcal C$ the following diagram commutes:
\begin{equation}\label{cd:Lax}
\begin{tikzcd}
\big(F(X)\otimes' F(Y)\big)\otimes' F(Z)
  \arrow[rr,"\alpha'_{F(X),F(Y),F(Z)}","\simeq"']
  \arrow[d, "{J_{X,Y}}\otimes'{\rm id}"']
&&F(X)\otimes'\big(F(Y)\otimes' F(Z)\big)
  \arrow[d, "{\rm id}\otimes'{J_{Y,Z}}"]
\\
F(X\otimes  Y)\otimes' F(Z)
  \arrow[d, "{J_{X\otimes Y,Z}}"']
&&F(X)\otimes'F(Y\otimes  Z)
  \arrow[d, "{J_{X,Y\otimes Z}}"]
\\
F\big((X\otimes  Y)\otimes Z\big)
  \arrow[rr,"F(\alpha_{X,Y,Z})","\simeq"']
&&F\big(X\otimes (Y\otimes Z)\big)
\end{tikzcd}
\end{equation}

\item[b.]
({\it unitality})
For all $X\in\mathcal C$ the following diagrams commute:
\[
\begin{tikzcd}
\mathsf 1'\otimes'F(X)
  \arrow[r,"\epsilon\otimes'{\rm id}"] 
  \arrow[d,"l_{F(X)}'"'] 
&F(\mathsf 1)\otimes'F(X)\arrow[d,"{J_{\mathsf 1,X}}"]
\\
F(X)
&F(\mathsf 1\otimes X)
  \arrow[l,"F(l_X)"] 
\end{tikzcd}
\quad
\begin{tikzcd}
F(X)\otimes'\mathsf 1'
  \arrow[r,"{\rm id}\otimes'\epsilon"] 
  \arrow[d,"r_{F(X)}'"'] 
&F(X)\otimes'F(\mathsf 1)
  \arrow[d,"{J_{X,\mathsf 1}}"]
\\
F(X)
&F(X\otimes \mathsf 1)
  \arrow[l,"F(r_X)"] 
\end{tikzcd}
\]
where $l$, $l'$, $r$ and $r'$ denote
the left and right unit constraints of the two monoidal categories~$\mathcal C$ and~$\mathcal C'$, respectively.
\end{enumerate}
\end{definition}

A natural transformation between two lax monoidal functors
is just a natural transformation respecting the monoidal structure.

\begin{definition}\label{def:MonoidalNatTransform}
Let $(F,J,\epsilon)$ and $(G,I,\kappa)$ be lax monoidal functors from
the monoidal category $(\mathcal C, \otimes, \mathsf1, \alpha, \iota)$
to the monoidal category
$(\mathcal C', \otimes', \mathsf1', \alpha', \iota')$.
A monoidal natural transformation $\nu$ from $(F,J,\epsilon)$ to $(G,I,\kappa)$
is a natural transformation $\nu_X\colon F(X)\to G(X)$ between the underlying functors
that is compatible with the tensor products and the units
in that categories, that is,
the following diagrams commute for all objects $X,Y\in\mathcal C$:
\[
\begin{tikzcd}
F(X)\otimes'F(Y)
  \arrow[r,"\nu_X\otimes'\nu_Y"]
  \arrow[d,"{J_{X,Y}}"']
&G(x)\otimes'G(Y)
  \arrow[d,"{I_{X,Y}}"]
\\
F(X\otimes Y)\arrow[r,"\nu_{X\otimes Y}"']
&G(X\otimes Y)
\end{tikzcd}
\qquad
\begin{tikzcd}
&\mathsf 1'\arrow[dl,"\epsilon"']\arrow[dr,"\kappa"] 
\\
F(\mathsf 1)\arrow[rr,"\nu_{1}"']&& G(\mathsf 1)
\end{tikzcd}
\]
\end{definition}

\medskip\par\noindent
{\it $G$-graded vector spaces and monoidal structures on the identity endofunctor.}
Let $G$ be a group and $\mathbb F$ be any field.
Consider the category ${\rm Vec}^G$ of $G$-graded vector spaces over $\mathbb F$.
Morphisms in this category are linear maps which preserve the $G$-degree.
The tensor product $V\otimes W$ on this category is defined by the formula
\begin{gather*}
(V\otimes W)_g:=\bigoplus_{x,y\in G\colon xy=g}V_x\otimes W_y.
\end{gather*}
The unit object $\mathsf 1$ is defined by $\mathsf 1_e=\mathbb F$ and $\mathsf 1_g=0$ for $g\ne e$.
The associator $\alpha$ and the isomorphism $\iota$ are inherited from vector spaces.
This endows the category ${\rm Vec}^G$ with a structure of an $\mathbb F$-linear monoidal category.

The simple objects in the category ${\rm Vec}^G$ are one-dimensional vector spaces $\delta_g$
indexed by the elements of the group $G$.
The $G$-grading on $\delta_g$ is given by $(\delta_g)_x=\mathbb F$ if $x=g$ and $(\delta_g)_x=0$ otherwise.
For these objects, we have $\delta_g\otimes\delta_h\simeq\delta_{gh}$.

The identity endofunctor $\mathrm{Id}\colon {\rm Vec}^G\to{\rm Vec}^G$ can be given the structure of a monoidal functor as follows.
Given a function $\sigma\colon G\times G\to\mathbb F^\times$,
we take as $\epsilon$ the identity of $\mathsf 1$ define a natural isomorphism $J^\sigma_{V,W}\colon V\otimes W\to V\otimes W$,
\[
J^{\sigma}_{V,W}(v\otimes w)=\sigma(g,h)v\otimes w\quad
\mbox{for all}\ v\in V_g,\ w\in W_h\quad
\mbox{and}\quad
 g,h\in G.
\]
Then diagram~\eqref{cd:Lax} commutes if and only if
\[
\sigma(g,hk)\sigma(h,k)=\sigma(k,gh)\sigma(g,h)\quad
\mbox{for all}\quad
g,h,k\in G,
\]
which means that $\sigma$ is a two-cocycle of $G$ with coefficients in $\mathbb F^\times$
with trivial $G$-action.
We denote the monoidal functor $(\mathrm{Id},J^\sigma)$ by $F_\sigma$.

An algebra object (not necessarily associative) in ${\rm Vec}^G$ is a pair $(A,\mu)$,
where $A$ is an object and $\mu\colon A\otimes A\to A$ is a morphism in ${\rm Vec}^G$.
Of course, this is equivalent to saying that $A$ is a $G$-graded algebra.
The monoidal functor $F_\sigma=(\mathrm{Id},J^\sigma)$ allows us to twist algebra objects as follows:
$F_\sigma\colon(A,\mu)\mapsto(A,\mu_\sigma)$,
where $\mu_\sigma$ is the composition of $J^\sigma_{A,A}\colon A\otimes A\to A\otimes A$
and $F_\sigma(\mu)=\mu\colon A\otimes A\to A$.
More explicitly,
\[
\mu_\sigma(a\otimes b)=\sigma(g,h)\mu(a\otimes b)\quad
\mbox{for all}\ a\in A_g,\ b\in A_h\quad
\mbox{and}\quad
g,h\in G.
\]

Since we are interested in algebras that satisfy certain identities,
we will need an additional structure, namely,
a symmetric braiding on ${\rm Vec}^G$.
For a braiding to exist, $G$ has to be {\it abelian}.
Let $\beta\colon G\times G\to\mathbb F^\times$ be a bicharacter of $G$, 
meaning that 
\begin{gather*}
\beta(gh,k)=\beta(g,k)\beta(h,k)
\quad\mbox{and}\quad
\beta(g,hk)=\beta(g,h)\beta(g,k)\ \ 
\mbox{for all}\ \ g,h,k\in G.
\end{gather*}
Then $\beta$ defines a braiding on the monoidal category ${\rm Vec}^G$ as follows.
The natural isomorphism $c_{V,W}\colon V\otimes W\to W\otimes V$
is given by 
\[
c_{V,W}(v\otimes w)=\beta(g,h)w\otimes v\quad
\mbox{for all}\ v\in V_g,\ w\in W_h\quad
\mbox{and}\quad
g,h\in G,
\]
which is symmetric (i.e., satisfies $c_{W,V}\circ c_{V,W}={\rm id}_{V\otimes W}$)
if and only if $\beta$ is (multiplicatively) skew-symmetric:
\[
\beta(g,h)=\beta(h,g)^{-1}
\quad\mbox{for all}\
g,h\in G.
\]
We denote the resulting symmetric braided category by $({\rm Vec}^G,\beta)$.
In this category, one can define any variety of algebras
given by a set of multilinear polynomial identities
by replacing the usual action of the symmetric group of degree $n$
on $A^{\otimes n}$ by the action arising from the braiding.
In particular, Lie algebras in $({\rm Vec}^G,\beta)$ are known as {\it Lie color algebras with commutation factor~$\beta$}.

The functor $F_\sigma$ is a braided monoidal functor from $({\rm Vec}^G,\beta)$ to $({\rm Vec}^G,\beta')$,
where
\[
\beta'(g,h):=\frac{\sigma(g,h)}{\sigma(h,g)}\beta(g,h)
\quad\mbox{for all}\
g,h\in G.
\]
This is the basis of the so-called {\it Scheunert's trick}
of transforming Lie color algebras to Lie superalgebras,
see more details about this subject in Scheunert's seminal paper~\cite{sche1979a}
and in \cite{baht2002a,baht2001a,pass1998a,pop1997a}.
Note that if $\sigma$ is symmetric,
i.e., $\sigma(g,h)=\sigma(h,g)$,
then $F_\sigma$ is a braided monoidal endofunctor on $({\rm Vec}^G,\beta)$.
Hence if $A$ is an algebra in $({\rm Vec}^G,\beta)$
that belongs to a variety defined by a set of multilinear polynomial identities,
then $F_\sigma(A)$ belongs to the same variety.
In particular, Lie algebras are mapped to Lie algebras.

\medskip\par\noindent
{\it Graded contractions as Lax monoidal structures on the identity endofunctor.}
We equip the identity endofunctor $\mathrm{Id}$ with a lax monoidal structure.
For the morphism $\epsilon\colon\mathbb F\to\mathbb F$ we simply take the identity~${\rm id}_{\mathbb F}$.
Given a function $\varepsilon\colon G\times G\to\mathbb F$,
we define a natural transformation $J^\varepsilon_{V,W}\colon V\otimes W\to V\otimes W$,
\[
J^{\varepsilon}_{V,W}(v\otimes w)=\varepsilon(g,h)v\otimes w\quad
\mbox{for all}\ v\in V_g,\ w\in W_h\quad
\mbox{and}\quad
g,h\in G.
\]
Then diagram~\eqref{cd:Lax} commutes if and only if
\[
\varepsilon(g,hk)\varepsilon(h,k)=\varepsilon(k,gh)\varepsilon(g,h)\quad
\mbox{for all}\quad
g,h,k\in G.
\]
As above, we want the monoidal functor $F_\varepsilon:=(\mathrm{Id},J^\varepsilon,{\rm id}_{\mathbb F})$
to map Lie algebras to Lie algebras,
so we will further require $\varepsilon$ to be symmetric.
These two conditions make $\varepsilon$ a generic contraction in the sense of Proposition~\ref{prop:GenericContraction}.

\begin{remark}
If we drop the symmetry condition on $\varepsilon$,
the functor $F_\varepsilon$ is a braided lax monoidal functor
from $({\rm Vec}^G,\beta)$ to $({\rm Vec}^G,\beta')$
if and only if
\[
\beta'(g,h)\varepsilon(h,g)=\beta(g,h)\varepsilon(g,h)
\quad\mbox{for all}\ g,h\in G.
\]
Hence, $F_\varepsilon$ maps Lie algebras in $({\rm Vec}^G,\beta)$ to Lie algebras $({\rm Vec}^G,\beta')$.
\end{remark}

The following result shows that Conjecture~\ref{conj:WeimWoods}
holds when the isomorphism of contractions is replaced with natural isomorphism.

\begin{proposition}\label{prop:LaxWeimWoodsConj}
The lax monoidal functors $F_\varepsilon:=(\mathrm{Id},J^\varepsilon,{\rm id}_{\mathbb F})$ and $F_{\varepsilon'}:=(\mathrm{Id},J^{\varepsilon'},{\rm id}_{\mathbb F})$ are naturally isomorphic
if and only if $\varepsilon\sim_{\rm n}\varepsilon'$,
i.e., $\varepsilon'=\varepsilon\, d\alpha$
for some $\alpha\colon G\to\mathbb F^\times$.
\end{proposition}

\begin{proof}
Assume that there exists $\alpha\colon G\to\mathbb F^\times$ satisfying
\begin{equation}\label{eq:ContractionEquivalence}
\varepsilon(g,h)=\alpha(gh)^{-1}\alpha(g)\alpha(h)\varepsilon'(g,h).
\end{equation}
For each $V\in {\rm Vec}^G$ define an isomorphism $\nu_V\colon V\to V$ via $\nu_V(v):=\alpha(g)v$
for all $v\in V_g$ and $g\in G$.
Since any morphism $f\colon V\to W$ in ${\rm Vec}^G$ preserves $G$-degrees,
the diagram
\[
\begin{tikzcd}
V
  \arrow[r,"f"] 
  \arrow[d,"\nu_V"'] 
&W\arrow[d,"\nu_W"] 
\\
V\arrow[r,"f"'] 
&W
\end{tikzcd}
\]
commutes,
hence $\nu$ is a natural isomorphism from the identity endofunctor to itself.
We need to verify that~$\nu$ respects the lax structure.
The commutativity of the first diagram from Definition~\ref{def:MonoidalNatTransform},
applied to~$F_\varepsilon$ and~$F_{\varepsilon'}$,
\[
\begin{tikzcd}
V\otimes W
  \arrow[r,"\nu_V\otimes\nu_W"] 
  \arrow[d,"{J^\varepsilon_{V,W}}"'] 
&V\otimes W\arrow[d,"{J^{\varepsilon'}_{V,W}}"] 
\\
V\otimes W\arrow[r,"\nu_{V\otimes W}"'] 
&V\otimes W
\end{tikzcd}
\]
is equivalent to equation~\eqref{eq:ContractionEquivalence}.
The second diagram from Definition~\ref{def:MonoidalNatTransform}
commutes since $\epsilon$ in the identity.

Conversely, suppose we have a natural isomorphism $\nu\colon F_\varepsilon\to F_{\varepsilon'}$ of 
lax monoidal functors.
Consider the simple objects $\delta_g$ in ${\rm Vec}^G$.
Since $\nu_{\delta_g}\colon \delta_g\to \delta_g$ is an isomorphism,
there exists $\alpha(g)\in\mathbb F^\times$ such that $\nu_{\delta_g}(v)=\alpha(g)v$
for all $v\in \delta_g$.
In view of this, the commutativity of the diagram
\[
\begin{tikzcd}
\delta_g\otimes\delta_h
  \arrow[r,"\nu_{\delta_g}\otimes\nu_{\delta_h}"] 
  \arrow[d,"{J^\varepsilon_{\delta_g,\delta_h}}"'] 
&\delta_g\otimes\delta_h\arrow[d,"{J^{\varepsilon'}_{\delta_g,\delta_h}}"] 
\\
\delta_g\otimes\delta_h\arrow[r,"\nu_{\delta_g\otimes\delta_h}"'] 
&\delta_g\otimes\delta_h
\end{tikzcd}
\]
is equivalent to equation~\eqref{eq:ContractionEquivalence}.
\end{proof}

\section{Conclusion}\label{sec:Conclusion}

Graded contractions of Lie algebras had been studied intensively over the past three decades, mostly from computational perspective.
All contractions had been computed and classified for some specific graded Lie algebras of low dimension, see, e.g.,
\cite{cout1991a,drap2024a,drap2024b,hriv2006a,hriv2013a} and references therein.
To the best of our knowledge, only Moody and Patera in~\cite{mood1991a} and Weimar--Woods in~\cite{weim2006a}
studied generic graded contractions from theoretical point of view.
We continued this direction but using different methods.

More specifically, Moody and Patera studied graded contractions from the perspectives
of semigroup cohomology and endofunctors of the category of Lie algebras,
see more details in Sections~\ref{sec:GroupCohomology} and~\ref{sec:EquivOfContractions}.
Weimar--Woods, on the other hand, developed a theory
based on the properties of what she called ``surviving defining relations'' and ``surviving higher-order identities'',
seemingly unrelated to cohomology.
This led to a classification theorem for generic graded contractions~\cite[Theorem~7.1]{weim2006a}
over the fields of complex and real numbers.

The works of Moody, Patera and Weimar--Woods inspired us to consider generic graded contractions
from the perspectives of group cohomology, affine algebraic geometry and monoidal categories.
These connections provided a wider context
to the classification problem of generic graded contractions,
allowing us to obtain streamlined proofs and generalize the results of 
\cite{mood1991a} and \cite{weim2006a}.\looseness=-1

Now we summarize the results of this paper and state some open problems.
We introduced the abelian group ${\rm H}^2_S(A)={\rm Hom}(\mathcal F_S/\langle\mathcal R_S\rangle,A)/{\rm B}^2_S(A)$,
which in the case $A=\mathbb F^\times$ and  $S\in\mathcal S(G)$
classifies the $\mathbb F$-valued contractions with support~$S$
up to equivalence via normalization.
In Theorem~\ref{thm:ClassificationOfContractions}, we presented exact sequences that
allow one to find the group ${\rm H}^2_S(A)$.
The special cases when $\mathbb F$ is an algebraically closed
or a real closed field were presented in Corollaries~\ref{cor:AlgClosedH^2_S} and~\ref{cor:RealClosedH^2_S}, respectively.
These corollaries were related to Theorem~7.1 from~\cite{weim2006a} in Remark~\ref{rem:SignInvariants},
thereby providing a cohomological interpretation of the notion of sign invariants from~\cite{weim2006a}.

Although we established the classification theorem (Theorem~\ref{thm:ClassificationOfContractions}) for generic graded contractions,
there is still an important question to address:
\begin{question}
Develop an efficient algorithm for finding the set of supports of generic graded contractions, $\mathcal S(G)$,
for any finite abelian group $G$.
\end{question}
\noindent
We touched upon this problem in Remark~\ref{rem:MeetSemilattice} and in Section~\ref{sec:VarietyOfContractions}.

We established in Corollary~\ref{cor:F_G/R_GIsFree} 
that $\mathcal F_G/\langle\mathcal R_G\rangle$ is a free abelian group
of rank $\leqslant |G|$.
More generally, the freeness of the group $\mathcal F_S/\langle\mathcal R_S\rangle$ is 
claimed in~\cite[Theorem~6.5]{weim2006a},
but the proof appears to be incorrect (see Remark~\ref{eq:F_S/R_SIsFree}).
Nevertheless, the numerical experiments that we carried out for low-order grading groups
lead us to pose the following question:
\begin{question}
Is it true that, for any $S\subset\mathcal P_G$,
the abelian group $\mathcal F_S/\langle\mathcal R_S\rangle$ is free?
\end{question}

As mentioned above, there are many works studying graded contractions of a particular $G$-graded Lie algebra.
We suggest studying generic graded contractions for particular grading groups or their classes
instead of fixing the Lie algebra:
\begin{question}
Describe the group ${\rm H}^2_S(\mathbb F^\times)$
for certain groups or classes of abelian groups, e.g., cyclic groups, elementary $p$-groups, etc.
\end{question}

In Section~\ref{sec:VarietyOfContractions}, we studied the affine algebraic variety $X$
of $G$-graded contractions and its cellular decomposition.
In Theorem~\ref{thm:VanishingIdeal}, we described the vanishing ideal
of the closure of the ${\rm T}_S$-orbit of an arbitrary generic graded contraction with support $S$.
This allowed us to relate graded contractions to graded degenerations, introduced in Section~\ref{sec:ContAndDegenOfLieAlgs},
and to In\"on\"u--Wigner contractions.
We also described the irreducible components
of the varieties of generic graded contractions of the grading groups $\mathbb Z_2$ and $\mathbb Z_3$,
see Examples~\ref{ex:Z2} and~\ref{ex:Z3}.
This leads us to the following question:
\begin{question}
How does the structure of the variety $X$ of generic graded contractions depend on the structure of the grading group $G$?
\end{question}

Since the ideal defining $X$ is generated by binomials,
it follows from~\cite{eise1996a} that the vanishing ideals of $X$ and its irreducible components
are also binomial, which implies that the irreducible components of $X$ are affine toric varieties, see, e.g.,~\cite{cox2011A}.
(A particular case is the irreducible component $X(\bar1)$ consisting of continuous graded contractions,
see Proposition~\ref{prop:X(1)Irred}.)
Moreover, $X$ is defined by homogeneous polynomials, so we can consider the corresponding projective variety and its components.
This observation paves the way to using a wide range of methods from toric geometry, such as combinatorics of polyhedra,
for further study of the components of~$X$.

Finally, in Section~\ref{sec:EquivOfContractions}, we revisited Conjecture~\ref{conj:WeimWoods} from~\cite{weim2006a}.
Interpreting generic $G$-graded contractions
as lax monoidal structures on the identity endofunctor
of the monoidal category of $G$-graded vector spaces,
we established a functorial version of this conjecture, but the original remains open.

\subsection*{Acknowledgments}
The authors are grateful to the anonymous reviewer whose feedback helped improve the exposition.
MVK is supported by NSERC Discovery Grant 2018-04883.
SDK is grateful to Dr.~Alex Bihlo for the support of his research thanks to funding from the Canada Research Chairs program,
the InnovateNL Leverage R{\&}D Program and the NSERC Discovery Grant.
SDK also acknowledges support from Dr.~Yuri Bahturin’s NSERC Discovery Grant.

\end{document}